\documentclass[final,3p,times]{elsarticle}

\usepackage{amsthm,amsmath,amssymb,mathrsfs,amsfonts,graphicx,graphics,latexsym,exscale,cmmib57,dsfont,amscd,ulem}
\usepackage{mathrsfs}
\usepackage{fancybox}
\usepackage{color,soul}
\usepackage[colorlinks=true]{hyperref}

\newcommand{\Z}{\boldsymbol{Z}}
\newcommand{\X}{\boldsymbol{X}}
\newcommand{\Y}{\boldsymbol{Y}}

\newcommand{\B}{\boldsymbol{B}}

\newcommand{\bea}{\begin{eqnarray}}
\newcommand{\eea}{\end{eqnarray}}
\newcommand{\bean}{\begin{eqnarray*}}
\newcommand{\eean}{\end{eqnarray*}}
\newtheorem*{Fthm}{Furstenberg's Structure Theorem}
\newtheorem*{FK}{Furstenberg-Katznelson Structure Theorem}
\newtheorem*{Thm0.2}{Theorem~0.2}

\newtheorem{Thm}{Theorem}[section]
\newtheorem{cor}[Thm]{Corollary}
\newtheorem{prop}[Thm]{Proposition}
\newtheorem{Lem}[Thm]{Lemma}

\theoremstyle{definition}
\newtheorem*{rem}{Remark}
\newtheorem{defn}[Thm]{Definition}
\newtheorem*{defn0.1}{Definition~0.1}
\newtheorem{remark}[Thm]{Remark}

\newtheorem*{Step1}{Step 1}
\newtheorem*{Step2}{Step 2}
\newtheorem*{Claim}{Assertion}

\numberwithin{equation}{section}

\journal{xxx}

\begin{document}
\begin{frontmatter}

\title{An extension of Furstenberg's structure theorem for Noetherian modules and multiple recurrence theorems I}

\author{Xiongping Dai}
\ead{xpdai@nju.edu.cn}
\address{Department of Mathematics, Nanjing University, Nanjing 210093, People's Republic of China}

\begin{abstract}
We extend Furstenberg's structure theorem to any standard Borel $G$-space, where $G$ is any locally compact second countable Noetherian module over a syndetic ring.
\end{abstract}

\end{frontmatter}

\section*{0.\ Introduction}\label{sec0}
We start this paper with recalling the celebrated Furstenberg structure theorem, which asserts the following statement:

\begin{Fthm}[1963~\cite{F63}]
Let $G\curvearrowright X$ be a topologically action of a locally compact group $G$ on a compact metric space $X$, which is minimal distal. Then there exists an ordinal $\eta$ such that to each ordinal $\xi\le\eta$ there is associated a factor $\pi_\xi\colon (X,G)\rightarrow(X_\xi,G)$ so that the followings are satisfied:
\begin{enumerate}
\item[$(a)$] $(X_0,G)$ is the one-point $G$-system and $(X_\eta,G)=(X,G)$.
\item[$(b)$] If $0\le\xi<\theta\le\eta$, then there is a factor map $\pi_{\theta,\xi}\colon(X_{\theta},G)\rightarrow(X_\xi,G)$ with $\pi_\xi=\pi_{\theta,\xi}\circ\pi_\theta$.
\item[$(c)$] For each $\xi<\eta$, $\pi_{\xi+1,\xi}\colon(X_{\xi+1},G)\rightarrow(X_\xi,G)$ is an isometric extension.
\item[$(d)$] If $\theta$ is a limit ordinal $\le\eta$, then $\underleftarrow{\lim}_{\xi<\theta}(X_\xi,G)$.
\end{enumerate}
\end{Fthm}
\noindent
The above statement itself and Furstenberg's original proof both are purely topological, not involving any probability theory~\cite{F63}. There are other structure theorems to attempt to generalize or simply prove Furstenberg's in the literature (cf.,~e.g., \cite{Vee,Ell,Z76,Zim82,AT,BF,NZ,NZ02}).
Particularly, we will be concerned with the following important and convenient version of Furstenberg structure theorem:

\begin{FK}[1978 \cite{FK}]
Let $(X,\mathscr{B}_X,\mu)$ be a regular $\mathbb{Z}^d$-space, where $1\le d<\infty$. Then there exists an ordinal $\eta$ and a system of $\mathbb{Z}^d$-factors
\begin{gather*}
\big{\{}\pi_\xi\colon\X=(X,\mathscr{B}_X,\mu,\mathbb{Z}^d)\rightarrow\X_\xi=(X_\xi,\mathscr{X}_\xi,\mu_\xi,\mathbb{Z}^d)\big{\}}_{\xi\le\eta}
\end{gather*}
such that:
\begin{enumerate}
\item[$(1)$] $\X_0$ is the one-point $\mathbb{Z}^d$-system and $\X_\eta=\X$ ($\mu$-$\mathrm{mod}$ $0$).
\item[$(2)$] If $\theta<\xi\le\eta$, then there is a factor $\mathbb{Z}^d$-map $\pi_{\xi,\theta}\colon\X_\xi\rightarrow\X_\theta$ with $\pi_\theta=\pi_{\xi,\theta}\circ\pi_\xi$.
\item[$(3)$] For each $\xi<\eta$, $\pi_{\xi+1,\xi}\colon\X_{\xi+1}\rightarrow\X_\xi$ is a primitive extension.
\item[$(4)$] If $\xi$ is a limit ordinal $\le\eta$, then $\X_\xi=\underleftarrow{\lim}_{\theta<\xi}\X_\theta$.
\end{enumerate}
We refer to this system of factors
$$\X\rightarrow\dotsm\rightarrow\X_{\xi+1}\rightarrow\X_\xi\rightarrow\dotsm\rightarrow\X_1\rightarrow\X_0,$$
possibly transfinite, as a ``composition factors series'' of the regular $\mathbb{Z}^d$-system $\X$.
\end{FK}

Since $\mathbb{R}$ is not a free abelian group, the above two theorems do not include the very important $C^0$-flow.
This paper will be devoted to developing Furstenberg Theory for measure-preserving dynamical systems of Noetherian-module actions far beyond the $\mathbb{Z}$- or $\mathbb{Z}^d$-spaces.

\subsection*{0.1.\ Noetherian modules and syndetic rings}\label{sec0.1}
Let $(R,+,\cdot)$ be a (not necessarily commutative) ring with the zero element $0$ for the commutative addition $+$, the identity element $1$ for the (not necessarily commutative) multiplication $\cdot$.

Let $(G,\circ)$, with the identity $I$, be an \textit{$R$-module}. By this we mean a group such that there exists a (left) scalar multiplication $(t,g)\mapsto tg$ from $R\times G$ to $G$
satisfying the property that for any two elements $S,T$ in $G$ and any two scalars $r,t$ in $R$,
\begin{itemize}
\item $t(S\circ T)=(tS)\circ(tT)$,\quad $(r+t)S=rS\circ tS$.
\end{itemize}
It is easy to check that any $R$-module must be an abelian group. $(R,+)$ itself, by letting $(G,\circ)=(R,+)$ with the identity $I=0$, is of course an $R$-module. See \cite[$\S$III.1]{Lan}.

By a topological group we here mean a Hausdorff space with a continuous group structure. Further, an $R$-module $G$ is referred to as a \textit{locally compact second countable} (\textit{lcsc}) $R$-module in this paper if
\begin{itemize}
\item $(R,+)$ and $(G,\circ)$ both are locally compact second countable groups and
$(t,g)\mapsto tg$ is continuous from the product space $R\times G$ to the space $G$.
\end{itemize}
Moreover $(R,+,\cdot)$ is called an \textit{lcsc ring} if $(R,+)$ itself is a locally compact second countable $R$-module.

We note here that any \textit{lcsc} abelian group $(G,\circ)$ is an \textit{lcsc} $\mathbb{Z}$-module by defining the scalar multiplication of $\mathbb{Z}\times G$ to $G$: $(n,g)\mapsto g^n=g\circ\dotsm\circ g$ ($n$-powers), where $\mathbb{Z}$ is endowed with the discrete topology.

Let $G$ be an \textit{lcsc} $R$-module. By a \textit{standard Borel $G$-space}, we mean a standard Borel probability space $(X,\mathscr{X},\mu)$ together with a Borel $G$-action from left on $X$ by measure-preserving transformations of $(X,\mathscr{X},\mu)$; that is to say, every $g\in G$ is a measure-preserving transformation of $(X,\mathscr{X},\mu)$ to itself and the $G$-action map $(g,x)\mapsto g(x)$ is Borel measurable of $G\times X$ to $X$. We will simply write $\X=(X,\mathscr{X},\mu,G)$ later on. We should bear in mind that different group elements may define the same $\mu$-preserving transformation of $X$.

To precisely formulate our Furstenbergwise  structure theorem we will prove in this paper, we first need to introduce the basic notions/conditions.

\begin{defn0.1}\label{def0.1}
\begin{enumerate}
\item[(1)]An $R$-module $G$ is said to be \textit{Noetherian} if it satisfies the ascending chain condition (ACC): for every sequence $G_0\subseteq G_1\subseteq G_2\subseteq\dotsm$ of $R$-submodules of $G$ we have $G_n=G_{n+1}$ for $n$ sufficiently large.
The ring $(R,+,\cdot)$ itself is called Noetherian if $(R,+)$ is just a Noetherian $R$-module over $(R,+,\cdot)$. See, e.g.,~\cite[$\S$VI.1]{Lan}.

\item[(2)]A subset $S$ of an \textit{lcsc} ring $(R,+,\cdot)$ is said to be \textit{syndetic} if one can find a compact subset $K$ of $R$ such that
$K+S=R$; cf.~e.g. \cite{GH, Fur}. Further $(R,+,\cdot)$ itself is called \textit{syndetic} if for each $r\in R$ with $r\not=0$, $rR$ is syndetic in $R$.
\end{enumerate}
\end{defn0.1}

For example, for $1\le n<\infty$ the $n$-dimensional lattice space $\mathbb{Z}^n$ thought of as a $\mathbb{Z}$-module and the $n$-dimensional euclidian space $\mathbb{R}^n$ as an $\mathbb{R}$-module both are Noetherian. In fact, every finitely generated module over a Noetherian ring is Noetherian (cf.~\cite[Proposition~VI.3]{Lan}). And by Hilbert's Basis Theorem it follows that the polynomial ring $\mathbb{F}[x_1,\dotsc,x_n]$ in $n$ variables over a field $\mathbb{F}$ is Noetherian and thus every finitely generated $\mathbb{F}[x_1,\dotsc,x_n]$-module is Noetherian (\cite[Theorem~VI.1]{Lan}).

Clearly, the integers ring $\mathbb{Z}$, the $p$-adic integers ring $\mathbb{Z}_p$, the rationals field $\mathbb{Q}$, the reals field $\mathbb{R}$ and the $p$-adic numbers field $\mathbb{Q}_p$ all are syndetic and Noetherian under the natural topologies and so all of the finitely generated topological $\mathbb{Q}$-modules, $\mathbb{R}$-modules, $\mathbb{Z}_p$-modules, and $\mathbb{Q}_p$-modules are Noetherian, which are \textit{lcsc} groups; see e.g. \cite{Lan}. Note that $(\mathbb{R}^n,+,\cdot)$ with $0=(0,\dotsc,0), 1=(1,\dotsc,1)$ is an \textit{lcsc} ring under the algebraic operations
\begin{gather*}
(x_1,\dotsc,x_n)+(y_1,\dotsc,y_n)=(x_1+y_1,\dotsc,x_n+y_n)\quad\textit{and}\quad (x_1,\dotsc,x_n)\cdot(y_1,\dotsc,y_n)=(x_1y_1,\dotsc,x_ny_n).
\end{gather*}
However, it is not a syndetic ring.
Particularly, it should be noted that $(\mathbb{R}^n,+)$ as an $\mathbb{R}$-module is Noetherian, but not as a $\mathbb{Z}$-module.

\subsection*{0.2.\ Structure theorem}\label{sec0.2}
Motivated by \cite{F63, Z76, FK, Fur},  in this paper we shall mainly prove the following Furstenberg-wise structure theorem including $C^0$-flow, which generalizes the FK structure theorem stated before.

\begin{Thm0.2}[Structure Theorem~I]\label{thm0.2}
Let $G$ be an \textit{lcsc} Noetherian $R$-module over a syndetic ring $(R,+,\cdot)$. Then for any nontrivial standard Borel $G$-space $\X$, there exists an ordinal $\eta$ and a system of $G$-factors $\left\{\pi_\xi\colon\X\rightarrow\X_\xi\right\}_{\xi\le\eta}$ such that:
\begin{enumerate}
\item[$(a)$] $\X_0$ is the one-point $G$-system and $\X_\eta=\X$ ($\mu$-$\mathrm{mod}$ $0$).
\item[$(b)$] If $0\le\theta<\xi\le\eta$, then there is a factor $G$-map $\pi_{\xi,\theta}\colon\X_\xi\rightarrow\X_\theta$ with $\pi_\theta=\pi_{\xi,\theta}\circ\pi_\xi$.
\item[$(c)$] For each ordinal $\xi$ with $0\le\xi<\eta$, $\pi_{\xi+1,\xi}\colon\X_{\xi+1}\rightarrow\X_\xi$ is a nontrivial ``primitive'' extension (cf.~Def.~\ref{def4.1} in $\S\ref{sec4.0}$).
\item[$(d)$] If $\xi$ is a limit ordinal $\le\eta$, then $\X_\xi=\underleftarrow{\lim}_{\theta<\xi}\X_\theta$.
\end{enumerate}
Moreover, the intermediate factors in our factors chain are of the form
\begin{gather*}
\X_\xi=(X,\mathscr{X}_\xi,\mu,G),\quad \{x\}\in\mathscr{X}_\xi\ \forall x\in X,\quad \pi_\xi=\textit{Id}_X\quad \textit{and}\quad \pi_{\xi+1,\xi}=\textit{Id}_X\quad (0<\xi<\eta).
\end{gather*}
We will call
\begin{gather*}
\X\rightarrow\dotsm\rightarrow\X_{\xi+1}\rightarrow\X_\xi\rightarrow\dotsm\rightarrow\X_1\rightarrow\X_0,
\end{gather*}
a ``Furstenberg factors chain'' of $\X$.
\end{Thm0.2}

This structure theorem claims that if regarding some dynamics, such as ``\textit{Sz}-'' and ``\textit{Kh}-properties'' we will consider in our subsequent paper~\cite{Dai-pre}, one interprets primitive extensions sufficiently broadly, then one can, by this procedure, describe all standard Borel $G$-systems acted by a Noetherian $R$-module.

As a byproduct Theorem~\ref{thm4.7} in $\S\ref{sec4.2}$ is another Zimmer-type structure theorem we will prove in this paper. The remainder of this paper will be organized as follows.

\tableofcontents
\subsection*{Acknowledgments}
Finally, the author is deeply grateful to Professor Hillel Furstenberg for many helpful suggestions, comments, and carefully checking the details of the original manuscript.

\section{Basic dynamics notions}\label{sec1}
In the sequel of this section, unless an explicit declaration, we let $G$ be an \textit{lcsc} $R$-module over any ring $(R,+,\cdot)$ not necessarily syndetic,
\begin{itemize}
\item $\X=(X,\mathscr{X},\mu,G)$ a standard Borel $G$-space so we may decompose $\mu$;
\end{itemize}
and let $(Y,\mathscr{Y},\nu)$ be another Borel $G$-space, not necessarily a standard Borel probability space, and simply write
\begin{itemize}
\item $\Y=(Y,\mathscr{Y},\nu,G)$.
\end{itemize}

We shall call $\X$ an \textit{extension} of $\Y$ or $\Y$ a \textit{factor} of $\X$ via a $G$-homomorphism $\pi$, usually written as $\pi\colon\X\rightarrow\Y$, if there is a measure-preserving map $\pi\colon(X,\mathscr{X},\mu)\rightarrow (Y,\mathscr{Y},\nu)$ satisfying
$\pi(g(x))=g(\pi(x))$ for all $g\in G$ and $x\in X$; i.e., the following commutative diagram holds:
$$
\begin{CD}
(X,\mathscr{X},\mu)@>{g}>>(X,\mathscr{X},\mu)\\
@V{\pi}VV @V{\pi}VV\\
(Y,\mathscr{Y},\nu)@>{g}>>(Y,\mathscr{Y},\nu)
\end{CD}\qquad \forall g\in G.
$$
Here $\pi$ is called a \textit{factor $G$-map} from $\X$ to $\Y$. Note here that $x\mapsto g(x)$ and $y\mapsto g(y)$ refer to the $G$-actions of the same $G$ on two different state spaces $X$ and $Y$.

An extension $\pi\colon\X\rightarrow\Y$ is said to be \textit{nontrivial} if $\pi^{-1}[\mathscr{Y}]\not=\mathscr{X}$ ($\mu\textrm{-mod }0$). In fact, $(Y,\mathscr{Y},\nu,G;\pi)=(X,\mathscr{Y},\mu,G;\textit{Id}_X)$ in our factors chain in Theorem~0.2.

For any $T\in G$, $\langle T\rangle_R$ will stand for the $R$-submodule of $G$ generated by the element $T$ over the same ring $(R,+,\cdot)$; that is to say,
\begin{gather*}
\langle T\rangle_R=\{tT; t\in R\}
\end{gather*}
that is the smallest $R$-submodule containing $T$ in $G$. Note that $\langle I\rangle_R=\{I\}$ and $0T=I$ for any $T\in G$, where we remind that $I$ is the identity of $G$.
\subsection{Totally relatively ergodic extensions}
Let $\pi\colon\X\rightarrow\Y$ be an extension. From now on, let $\mathscr{X}^\prime$ be a $\sigma$-subalgebra of $\mathscr{X}$, where $(X,\mathscr{X}^\prime,\mu)$ is not necessarily a standard Borel probability space, such that
\begin{itemize}
\item $\pi^{-1}[\mathscr{Y}]\subseteq\mathscr{X}^\prime$, $\mathscr{X}^\prime$ is $G$-invariant (i.e. $g^{-1}[\mathscr{X}^\prime]\subseteq\mathscr{X}^\prime\ \forall g\in G$),
\end{itemize}
and write
\begin{itemize}
\item $\X^\prime=(X,\mathscr{X}^\prime,\mu, G)$.
\end{itemize}
Then $\X^\prime$ is a factor of $\X$ via the $G$-map $\textit{Id}_X\colon X\rightarrow X;\ x\mapsto x$ and it is also an extension of $\Y$ via the same $\pi\colon X\rightarrow Y$. We will call $\X\xrightarrow[]{\textit{Id}_X}\X^\prime\xrightarrow[]{\pi}\Y$ a short factors series.

\begin{defn}\label{def1.1}
For $\X^\prime=(X,\mathscr{X}^\prime,\mu,G)$, we say the extension $\pi\colon\X^\prime\rightarrow\Y$ is to be
\begin{itemize}
\item \textit{relatively ergodic for an element $g$} in $G$ if every $\langle g\rangle_R$-invariant $f\in\mathfrak{L}^2(X,\mathscr{X}^\prime,\mu)$ is ($\mu$-\textit{a.e.}) a function on $\Y$ (or equivalently, every $\langle g\rangle_R$-invariant $f\in\mathfrak{L}^\infty(X,\mathscr{X}^\prime,\mu)$ is also $\pi^{-1}[\mathscr{Y}]$-measurable);
\begin{itemize}
\item further, to be \textit{totally relatively ergodic for an $R$-submodule $\varGamma$} of $G$ if it is relatively ergodic for each element $g$ in $\varGamma$ with $g\not=I$.
\end{itemize}
\item \textit{jointly relatively ergodic for an $R$-submodule $\varGamma$} of $G$ if every $\varGamma$-invariant $f\in\mathfrak{L}^\infty(X,\mathscr{X}^\prime,\mu)$ is ($\mu$-\textit{a.e.}) a function on $\Y$.
\end{itemize}
\end{defn}

Notice here that $\pi\colon\X^\prime\rightarrow\Y$ fails to be totally relatively ergodic for $G$ itself if some set in $\mathscr{X}^\prime$ that is not the preimage of a set in $\mathscr{Y}$ is $\langle T\rangle_R$-invariant for some $T\in G$ with $T\not=I$. Ordinarily when one speaks of ergodicity of a Borel $G$-space $(X,\mathscr{X}^\prime,\mu)$, one is concerned about $G$-invariance, not only $\langle T\rangle_R$-, of sets in $\mathscr{X}^\prime$.

Particularly, only from the viewpoint of group in \cite{FK,Fur}, $\pi\colon\X^\prime\rightarrow\Y$ fails to be (totally) relatively ergodic for $G$ if some set in $\mathscr{X}^\prime$ that is not the preimage of a set in $\mathscr{Y}$ is $T$ or equivalently $\langle T\rangle_\mathbb{Z}$, but not $\langle T\rangle_R$, -invariant for some $T\in G$ with $T\not=I$.

If $\Y$ is a one-point $G$-system, then the relative ergodicity of $\X^\prime$ for $g$ is just the classical ergodicity of the metric dynamical system $(X,\mathscr{X}^\prime,\mu, \langle g\rangle_R)$ for the $R$-submodule $\langle g\rangle_R$.
\subsection{Relative-product extensions}\label{sec1.2}
For any extension $\pi\colon\X\rightarrow\Y$ (not for $\pi\colon\X^\prime\rightarrow\Y$), by Doob's theorem, there exists a random measure on the standard Borel probability space $(X,\mathscr{X},\mu)$:
\begin{gather*}
\mu(\centerdot,\centerdot)\colon Y\times\mathscr{X}\rightarrow\mathbb{R}\quad\textrm{or write}\quad \{\mu_y\colon\mathscr{X}\rightarrow\mathbb{R}\}_{y\in Y}
\end{gather*}
such that for any $y\in Y, B\in\mathscr{X}$ and $\varphi\in\mathfrak{L}^1(X,\mathscr{X},\mu)$,
\begin{enumerate}
\item[(1)] $\mu_y(\centerdot)\colon\mathscr{X}\rightarrow\mathbb{R};\ A\mapsto\mu(y,A)$ is a probability measure on $(X,\mathscr{X})$;
\item[(2)] $\mu(\centerdot,B)\colon Y\rightarrow\mathbb{R};\ y\mapsto\mu(y,B)$ is a $\mathscr{Y}$-measurable function;
\item[(3)] $\mu_{\pi(x)}(B)=E_\mu(1_B|\pi^{-1}[\mathscr{Y}])(x)$ for $\mu$-\textit{a.e.} $x\in X$;
\item[(4)] $E_\mu(\varphi|\pi^{-1}[\mathscr{Y}])(x)=\int_X\varphi(z)\mu(\pi(x),dz)$ for $\mu$-\textit{a.e.} $x\in X$.\footnote{In fact, $(3)\Leftrightarrow(4)$. In addition, we set $E_\mu(\varphi|\Y)(\centerdot)=\int_X\varphi(x)\mu(\centerdot,dx)\in\mathfrak{L}^1(Y,\mathscr{Y},\nu)$ as in \cite[Chap.~5]{Fur}. In addition, if $(X,\mathscr{X},\mu)$ is not a standard Borel probability space, then such $\mu(\centerdot,\centerdot)$ does not need to exist (cf., e.g.,~\cite{Die}). This is just the reason why we will always regard $\X^\prime$ as a factor of $\X$, not individually an extension of $\Y$.}
\end{enumerate}
Here $\{\mu_y\}_{y\in Y}$ or write $\mu=\int_Y\mu_y\nu(dy)$ is called the \textit{disintegration} of $\mu$, over $\X\xrightarrow[]{\pi}\Y$, which is unique for $\nu$-\textit{a.e.} $y\in Y$. If $\mathscr{Y}$ is countably generated, then $\mu_y$ is supported on the $\pi$-preimage of the atom $[y]_\mathscr{Y}$ and $\mu_y=\mu_z$ for any $z\in[y]_\mathscr{Y}$.

Following \cite{F77, FK} and \cite[$\S5.5$]{Fur}, the \textit{condition-independent measure} $\mu\otimes_{\Y}\mu$ on $(X\times X,\mathscr{X}\otimes\mathscr{X})$ over $(\Y,\pi)$ is defined by the disintegration
\begin{gather*}
\mu\otimes_{\Y}\mu=\int_Y\mu_y\otimes\mu_y\nu(dy).
\end{gather*}
Then by \cite[Proposition~5.10]{Fur},
\begin{gather*}
\int_{X\times X}f_1\otimes f_2d\mu\otimes_{\Y}\mu=\int_YE_\mu(f_1|\Y)E_\mu(f_2|\Y)d\nu\quad\forall f_1,f_2\in\mathfrak{L}^2(X,\mathscr{X},\mu).
\end{gather*}
Clearly $\mu\otimes_{\Y}\mu$ is $G$-invariant where $G$ acts diagonally on $X\times X$ by the following standard way:
$$g(x,x^\prime)=(g(x),g(x^\prime))\quad \forall (x,x^\prime)\in X\times X\textrm{ and }g\in G;$$
hence $(X\times X,\mathscr{X}\otimes \mathscr{X},\mu\otimes_{\Y}\mu,G)$ is a measure-preserving $G$-system, which is called the \textit{relative-product extension of $(\Y,\pi)$}, write $(X,\mathscr{X},\mu,G)\times_{\Y}(X,\mathscr{X},\mu,G)$ or simply $\X\times_{\Y}\X$. Then $\{\mu_y\otimes\mu_y\}_{y\in Y}$ is exactly the classical disintegration of $\mu\otimes_{\Y}\mu$, over the factor $\Y$ via the $G$-map\footnote{It should be noted that $\pi_1^{-1}[A]=\pi_2^{-1}[A]\ (\mu\otimes_{\Y}\mu\textrm{-mod }0)$ for any $A\in\mathscr{Y}$; cf.~\cite[Proposition~5.11]{Fur}.}
\begin{gather*}
\pi\times_{\Y}\pi\colon (x,x^\prime)\mapsto\pi_1(x)=x\quad \textit{or}\quad \pi\times_{\Y}\pi\colon (x,x^\prime)\mapsto\pi_2(x^\prime)=x^\prime.
\end{gather*}
Then
\begin{gather*}
E_{\mu\otimes_{\Y}\mu}(f_1\otimes f_2|\Y)=E_\mu(f_1|\Y)E_\mu(f_2|\Y)\quad\forall f_1,f_2\in\mathfrak{L}^2(X,\mathscr{X},\mu).
\end{gather*}
See \cite[Propositions~5.12 and 5.14]{Fur}.


Let $\{U_g\colon \phi\mapsto\phi\circ g\}_{g\in G}$ be the naturally induced Koopman unitary operators on the following $\mathfrak{L}^2$-spaces of complex-valued functions:
\begin{gather*}
\mathscr{H}=\mathfrak{L}^2(X,\mathscr{X},\mu),\quad \mathscr{H}\otimes_{\Y}\mathscr{H}=\mathfrak{L}^2(X\times X,\mathscr{X}\otimes\mathscr{X},\mu\otimes_{\Y}\mu),\intertext{and}
\mathscr{H}_y=\mathfrak{L}^2(X,\mathscr{X},\mu_y),\quad \mathscr{H}_y\otimes_{\Y}\mathscr{H}_y=\mathfrak{L}^2(X\times X,\mathscr{X}\otimes\mathscr{X},\mu_y\otimes\mu_y),\quad \forall y\in Y.
\end{gather*}
For any $\varphi\in\mathscr{H}$ or $\mathscr{H}\otimes_{\Y}\mathscr{H}$, we denote its $\mathfrak{L}^2$-norm by $\|\varphi\|_2$ (or by $\|\varphi\|_{2,\mu}, \|\varphi\|_{2,\mu\otimes_{\Y}\mu}$); if $\varphi\in\mathscr{H}_y$ or $\mathscr{H}_y\otimes_{\Y}\mathscr{H}_y$, we denote its $\mathfrak{L}^2$-norm by $\|\varphi\|_{2,y}$ or $\|\varphi\|_{2,\mu_y\otimes\mu_y}$. We say that $\varphi\in\mathscr{H}$ on $X$ or $\mathscr{H}\otimes_{\Y}\mathscr{H}$ on $X\times X$ is \textit{fiberwise $\mathfrak{L}^2$-bounded} if $\|\varphi\|_{2,y}$ or $\|\varphi\|_{2,\mu_y\otimes\mu_y}$ is bounded as a function of $y$ on $Y$.

We note that if $\phi,\psi\in\mathscr{H}$ then $\phi\otimes\psi\in \mathfrak{L}^1(X\times X,\mathscr{X}\otimes\mathscr{X},\mu\otimes_{\Y}\mu)$, since $\|\phi\|_{2,y}$ and $\|\psi\|_{2,y}$ as functions of the variable $y$ on $Y$ are in $\mathfrak{L}^2(Y,\mathscr{Y},\nu)$ and
\begin{equation*}\begin{split}
\int_{X\times X}|\phi\otimes\psi(x,x^\prime)|\mu\otimes_{\Y}\mu(d(x,x^\prime))&=\int_Y\int_{X\times X}|\phi\otimes\psi(x,x^\prime)|\mu_y\otimes\mu_y(d(x,x^\prime))\nu(dy)\\
&=\int_Y\left(\int_X|\phi(x)|\mu_y(dx)\int_X|\psi(x^\prime)|\mu_y(x^\prime)\right)\nu(dy)\\
&\le\int_Y\|\phi\|_{2,y}\|\psi\|_{2,y}\nu(dy)\\
&\le\|\phi\|_{2,\mu}\|\psi\|_{2,\mu}
\end{split}\end{equation*}
by
\begin{equation*}
\|\phi\|_{2,\mu}=\left(\int_Y\|\phi\|_{2,y}^2\nu(dy)\right)^{\frac12}\quad \textrm{and}\quad \|\phi\|_{2,y}=\left(E_\mu(|\phi|^2|\Y)(y)\right)^{\frac12}=\left(\int_X|\phi(x)|^2\mu_y(dx)\right)^{\frac12}
\end{equation*}
for $\nu$-\textit{a.e.} $y\in Y$.

Since $(X,\mathscr{X})$ is a standard Borel $G$-space and $G$ is \textit{lcsc}, by Varadarajan's isomorphism theorem~\cite[Theorem~3.2]{Var} or \cite[Theorem~2.1.19]{Zim} that is independent of the module structure of $G$, it follows that $(X,\mathscr{X})$ is $G$-isomorphic to some $G$-invariant Borel subset of certain compact metric $G$-space. Thus combined with the $G$-invariance of $\mu$ this yields the well-known fact that the function $g\mapsto U_g\phi$ of $G$ into $\mathscr{H}$ is continuous under the $\mathfrak{L}^2$-norm of $\mathscr{H}$ with respect to $\mu$, for any given $\phi$ in $\mathscr{H}$.

\begin{rem}
Given any short factors series $\X\xrightarrow[]{\textit{Id}_X}\X^\prime=(X,\mathscr{X}^\prime,\mu,G)\xrightarrow[]{\pi}\Y$, we will write
\begin{gather*}
\mathscr{H}^\prime=\mathfrak{L}^2(X,\mathscr{X}^\prime,\mu),\\
\mathscr{H}^\prime\otimes_{\Y}\mathscr{H}^\prime=\mathfrak{L}^2(X\times X,\mathscr{X}^\prime\otimes\mathscr{X}^\prime,\mu\otimes_{\Y}\mu),\intertext{and for all $y\in Y$}
\mathscr{H}_y^\prime=\mathfrak{L}^2(X,\mathscr{X}^\prime,\mu_y),\\
\mathscr{H}_y^\prime\otimes_{\Y}\mathscr{H}_y^\prime=\mathfrak{L}^2(X\times X,\mathscr{X}^\prime\otimes\mathscr{X}^\prime,\mu_y\otimes\mu_y).
\end{gather*}
Then, $\mathscr{H}^\prime, \mathscr{H}^\prime\otimes_{\Y}\mathscr{H}^\prime, \mathscr{H}_y^\prime$, and $\mathscr{H}_y^\prime\otimes_{\Y}\mathscr{H}_y^\prime$ are closed subspaces of $\mathscr{H}, \mathscr{H}\otimes_{\Y}\mathscr{H}, \mathscr{H}_y$, and $\mathscr{H}_y\otimes_{\Y}\mathscr{H}_y$, respectively.
It should be noted, however, that $\{\mu_y\}_{y\in Y}$ and then $\mu\otimes_{\Y}\mu$ are induced by $\pi\colon\X\rightarrow\Y$, never by $\pi\colon\X^\prime\rightarrow\Y$, for the intermediate factor $\X^\prime$ we consider later is not necessarily a standard Borel $G$-space.
\end{rem}
\subsection{Relative convolutions of functions and precompactness}\label{sec1.3}
In the sequel of this subsection, let there be any given a short factors series:
\begin{gather*}
\X\xrightarrow[]{\textit{Id}_X}\X^\prime=(X,\mathscr{X}^\prime,\mu,G)\xrightarrow[]{\pi}\Y.
\end{gather*}

Given any two functions $H\in\mathscr{H}^\prime\otimes_{\Y}\mathscr{H}^\prime$ and $\phi\in\mathscr{H}^\prime$, as in \cite{FK,Fur} for the special case $(R,+,\cdot)=(\mathbb{Z},+,\cdot)$, we now define the \textit{convolution}, written as $H*_{\Y}\phi$, of $H$ with $\phi$ relative to the factor $\pi\colon\X^\prime\rightarrow\Y$ by
\begin{equation}\label{eq1.1}
H*_{\Y}\phi(x)=\int_XH(x,x^\prime)\phi(x^\prime)\mu_{\pi(x)}(dx^\prime).
\end{equation}
Since for $\nu$-\textit{a.e.} $y\in Y$ here $H(x,x^\prime)$ is a function in $\mathscr{H}_y^\prime\otimes_{\Y}\mathscr{H}_y^\prime$, so that for $\mu_y$-\textit{a.e.} $x\in X$ the integrand $H(x,\centerdot)\phi(\centerdot)$ in (\ref{eq1.1}) is the product of two functions of the variable $x^\prime\in X$ in $\mathscr{H}_y^\prime$, and hence $H(x,\centerdot)\phi(\centerdot)\in \mathfrak{L}^1(X,\mathscr{X}^\prime,\mu_{\pi(x)})$ by Fubini's theorem and so the integral exists. Thus (\ref{eq1.1}) is well defined for $\mu$-\textit{a.e.} $x\in X$, for $\mu=\int_Y\mu_y\nu(dy)$. Moreover,
\begin{equation}\label{eq1.2}
\|H*_{\Y}\phi\|_{2,y}\le\|H\|_{2,\mu_y\otimes\mu_y}\|\phi\|_{2,y}\quad \nu\textit{-a.e. }y\in Y.
\end{equation}
From this we can conclude that if $H$ is fiberwise $\mathfrak{L}^2$-bounded (i.e. $\|H\|_{2,\mu_y\otimes\mu_y}\le M<\infty$ for $\nu$-\textit{a.e.} $y\in Y$), then
\begin{equation*}
\|H*_{\Y}\phi\|_2\le M\|\phi\|_2
\end{equation*}
so that $H*_{\Y}\phi\in\mathscr{H}^\prime$ and further
\begin{gather*}
H*_{\Y}\centerdot\colon\phi\mapsto H*_{\Y}\phi
\end{gather*}
is a bounded linear operator of $\mathscr{H}^\prime$ into itself. On the other hand, if $\phi\in \mathfrak{L}^\infty(X,\mathscr{X}^\prime,\mu)$, then for any $H\in\mathscr{H}^\prime\otimes_{\Y}\mathscr{H}^\prime$
\begin{equation}\label{eq1.3}
\begin{split}
\|H*_{\Y}\phi\|_2&\le\|\phi\|_\infty\left(\int_X\int_X|H(x,x^\prime)|^2\mu_{\pi(x)}(dx^\prime)\mu(dx)\right)^{1/2}\\
&=\|\phi\|_\infty\left(\int_Y\int_X\int_X|H(x,x^\prime)|^2\mu_{\pi(x)}(dx^\prime)\mu_y(dx)\nu(dy)\right)^{1/2}\\
&=\|\phi\|_\infty\left(\int_Y\left(\iint_{X\times X}|H|^2d\mu_y\otimes\mu_y\right)\nu(dy)\right)^{1/2}\\
&=\|\phi\|_\infty\|H\|_{2,\mu\otimes_{\Y}\mu}
\end{split}\end{equation}
so that
\begin{gather*}
\Box*_{\Y}\phi\colon H\mapsto H*_{\Y}\phi
\end{gather*}
is a bounded operator of $\mathscr{H}^\prime\otimes_{\Y}\mathscr{H}^\prime$ into $\mathscr{H}^\prime$.

As usual, a subset $A$ in a topological space $M$ is called \textit{precompact} if its closure $\overline{A}$ relative to $M$ is compact. Recall that a subset $A$ in a complete metric space $M$ is precompact \textsl{iff} for any $\delta>0$ there are points $a_1,\dotsc,a_k\in A$ such that
$A$ is covered by the $k$ balls of centered at $a_i$ of radius $\delta$ and \textsl{iff} for any $\delta>0$ there are points $a_1,\dotsc,a_k\in M$ such that
$A$ is covered by the $k$ balls of centered at $a_i$ of radius $\delta$. If $A$ is precompact in a Hausdorff space, then any subset of $A$ is also precompact.

The following is a criterion of precompactness of subset of $(\mathfrak{L}^2(X,\mathscr{X}^\prime,\mu),\|\cdot\|_{2,y})$, which is of interest itself.

\begin{Thm}\label{thm1.2}
Given any short factors series $\X\xrightarrow[]{\textit{Id}_X}\X^\prime\xrightarrow[]{\pi}\Y$, consider the relative convolution linear operator on $\left(\mathscr{H}^\prime\otimes_{\Y}\mathscr{H}^\prime\right)\times\mathscr{H}^\prime$ to $\mathscr{H}^\prime$:
\begin{gather*}
\Box{\Large{*}_{\Y}}\centerdot\colon \mathfrak{L}^2(X\times X,\mathscr{X}^\prime\otimes\mathscr{X}^\prime,\mu\otimes_{\Y}\mu)\times \mathfrak{L}^2(X,\mathscr{X}^\prime,\mu)\longrightarrow \mathfrak{L}^2(X,\mathscr{X}^\prime,\mu)
\end{gather*}
given by $(H,\phi)\mapsto H*_{\Y}\phi$ as in $(\ref{eq1.1})$; and for any $r>0$ and $H\in\mathscr{H}^\prime\otimes_{\Y}\mathscr{H}^\prime$, write
\begin{gather*}
\B_r^{2,\infty}=\left\{\phi\in \mathfrak{L}^2(X,\mathscr{X}^\prime,\mu)\colon\|\phi\|_\infty\le r\right\}\quad \textrm{and}\quad H*_{\Y}\B_r^{2,\infty}=\left\{H*_{\Y}\phi\,|\,\phi\in\B_r^{2,\infty}\right\}.
\end{gather*}
Then for any $H\in\mathscr{H}^\prime\otimes_{\Y}\mathscr{H}^\prime$ and $r>0$, the set $H*_{\Y}\B_r^{2,\infty}$ is precompact in the Hilbert space $(\mathfrak{L}^2(X,\mathscr{X}^\prime,\mu),\|\cdot\|_{2,y})$ for $\nu$-\textit{a.e.} $y\in Y$.
\end{Thm}

\begin{proof}
For our simplicity, assume all the $\mathfrak{L}^2$-spaces involved in the theorem consist of real-valued functions without loss of generality. Let $r>0$ and $H\in\mathscr{H}^\prime\otimes_{\Y}\mathscr{H}^\prime$ be any given. We will divide our proof of Theorem~\ref{thm1.2} into two steps.

\begin{Step1}
If $H$ is of the form $H(x,x^\prime)=\sum_{j=1}^J\psi_j(x)\psi_j^\prime(x^\prime)$ with $\psi_j(x),\psi_j^\prime(x^\prime)\in \mathfrak{L}^\infty(X,\mathscr{X}^\prime,\mu)$ and $1\le J<\infty$, then $H*_{\Y}\B_r^{2,\infty}$ is precompact in the Hilbert space $(\mathfrak{L}^2(X,\mathscr{X}^\prime,\mu),\|\cdot\|_{2,y})$ for $\nu$-\textit{a.e.} $y\in Y$.
\end{Step1}
\begin{proof}
To prove the above claim, we first note that by the Banach-Alaoglu Theorem, the closed ball of radius $r$
\begin{gather*}
E_y=\left\{\phi(x^\prime)\in \mathfrak{L}^2(X,\mathscr{X}^\prime,\mu_y)\colon \|\phi\|_{2,y}\le r\right\}
\end{gather*}
is weakly compact in $\mathfrak{L}^2(X,\mathscr{X}^\prime,\mu_y)$ for each $y\in Y$. Moreover, since $(X,\mathscr{X})$ is a standard Borel space, hence $\mathscr{H}_y$ and further $\mathscr{H}_y^\prime$ are separable. This implies that $E_y$ is metrizable with the weak topology.
Then by the Tychonoff Theorem it follows that $\prod_{y\in Y}E_y$ is a compact convex subset of the locally-convex Hausdorff product-topological vector space
$${\prod}_{y\in Y}\mathfrak{L}^2(X,\mathscr{X}^\prime,\mu_y)$$
where each of the factor spaces is under the natural weak topology.\footnote{Any separable Hilbert space is locally convex under the weak topology.}

Let $\{\phi_n\}_{n=1}^\infty$ be an arbitrary sequence of functions in $\B_r^{2,\infty}$. By dropping a zero-measure set if necessary, we may assume $|\phi_n(x^\prime)|\le r$ for all $n\ge1$ and any $x^\prime\in X$. Since $\phi_n\in E_y$ for each $y\in Y$, hence $\langle\phi_{n,y}\rangle_{y\in Y}\in\prod_{y\in Y}E_y$ where $\phi_{n,y}(x^\prime)=\phi_n(x^\prime)$ for all $n$. Therefore by \cite[Lemma~2]{Dai-f}, from the vector sequence $\left\{\langle\phi_{n,y}\rangle_{y\in Y}\right\}_{n=1}^\infty$ may be extracted a subsequence $\left\{\langle\phi_{n_{\ell},y}\rangle_{y\in Y}\right\}_{\ell=1}^\infty$ such that as $\ell\to\infty$ we have
\begin{gather*}
\phi_{n_\ell,y}\xrightarrow[]{\textit{weakly in }\mathfrak{L}^2(X,\mathscr{X}^\prime,\mu_y)}\phi_y\quad \forall y\in Y
\end{gather*}
for some point $\langle\phi_{y}\rangle_{y\in Y}\in\prod_{y\in Y}E_y$. Given any $\psi\in\mathfrak{L}^\infty(X,\mathscr{X}^\prime,\mu)$, $y\mapsto\langle\psi,\phi_y\rangle_y$ is $\mathscr{Y}$-measurable from $Y$ to $\mathbb{R}$ and so $x\mapsto\langle\psi,\phi_{\pi(x)}\rangle_{\pi(x)}$ is $\mathscr{X}^\prime$-measurable from $X$.



Therefore for any function $H(x,x^\prime)=\sum_{j=1}^J\psi_j(x)\psi_j^\prime(x^\prime)$ with $\psi_j(x),\psi_j^\prime(x^\prime)$ in $\mathfrak{L}^\infty(X,\mathscr{X}^\prime,\mu)$ and $1\le J<\infty$, it holds that
$$
\sum_{j=1}^J\psi_j(x)\langle\psi_j^\prime,\phi_{n_\ell}\rangle_{\mathscr{H}_{\pi(x)}^\prime}=\sum_{j=1}^J\psi_j(x)\langle\psi_j^\prime,\phi_{n_\ell,\pi(x)}\rangle_{\mathscr{H}_{\pi(x)}^\prime}
\xrightarrow{\mu_y\textit{-a.e. }x\textit{ in }X}\sum_{j=1}^J\psi_j(x)\langle\psi_j^\prime,\phi_{\pi(x)}\rangle_{\mathscr{H}_{\pi(x)}^\prime}
$$
as $\ell\to\infty$. This implies that for $\mu_y\textit{-a.e.~}x\in X$,
\begin{gather*}
H*_{\Y}\phi_{n_\ell}(x)\to H*_{\Y}\phi_y(x) \quad \textit{as }\ell\to\infty.
\end{gather*}
So by $|H*_{\Y}\phi_{n_\ell}(x)|\le r\|H\|_\infty<\infty$ and the Lebesque Dominated Convergence Theorem, we see that $H*_{\Y}\phi_{n_\ell}\to H*_{\Y}\phi_y$ in the Hilbert space $(\mathfrak{L}^2(X,\mathscr{X}^\prime,\mu), \|\cdot\|_{2,y})$.

This proves that $H*_{\Y}\B_r^{2,\infty}$ is precompact in the Hilbert space $(\mathfrak{L}^2(X,\mathscr{X}^\prime,\mu), \|\cdot\|_{2,y})$.
\end{proof}

\begin{Step2}
Let $H$ be arbitrarily given in $\mathfrak{L}^2(X\times X,\mathscr{X}^\prime\otimes\mathscr{X}^\prime,\mu\otimes_{\Y}\mu)$; then $H*\B_r^{2,\infty}$ is precompact in $(\mathfrak{L}^2(X,\mathscr{X}^\prime,\mu), \|\cdot\|_{2,y})$ for $\nu$-\textit{a.e.} $y\in Y$.
\end{Step2}

\begin{proof}
Indeed since $\mathfrak{L}^\infty(X,\mathscr{X}^\prime,\mu)\otimes \mathfrak{L}^\infty(X,\mathscr{X}^\prime,\mu)$ is dense in $\mathfrak{L}^2(X\times X,\mathscr{X}^\prime\otimes\mathscr{X}^\prime,\mu\otimes_{\Y}\mu)$, hence there exists $H_n\to H$ in $\mathfrak{L}^2(X\times X,\mathscr{X}^\prime\otimes\mathscr{X}^\prime,\mu\otimes_{\Y}\mu)$ where each $H_n(x,x^\prime)$ has the special form of finite linear combination $\sum_j\psi_j\otimes\psi_j^\prime$ with $\psi_j,\psi_j^\prime\in \mathfrak{L}^\infty(X,\mathscr{X}^\prime,\mu)$.
Since by (\ref{eq1.3}) we have
\begin{gather*}
\|H_n*_{\Y}\phi-H*_{\Y}\phi\|_2\le r\|H_n-H\|_{2,\mu\otimes_{\Y}\mu}\quad\forall \phi\in\B_r^{2,\infty},
\end{gather*}
we see that as $n\to\infty$, under the $\mathfrak{L}^2$-norm $\|\cdot\|_2$
\begin{gather*}
H_n*_{\Y}\phi\to H*_{\Y}\phi
\end{gather*}
in $\mathfrak{L}^2(X,\mathscr{X}^\prime,\mu)$ uniformly for $\phi\in\B_r^{2,\infty}$. Let $\varepsilon>0$ be arbitrary. For $H_n*_{\Y}\B_r^{2,\infty}$ is precompact in $(\mathfrak{L}^2(X,\mathscr{X}^\prime,\mu),\|\cdot\|_{2,y})$ by Step~1, there exists a finite set of functions, say $\phi_1^{(n)},\dotsc,\phi_{k_n}^{(n)}$, in $\B_r^{2,\infty}$ so that
$$
\min_{1\le i\le k_n}\big{\|}H_n*_{\Y}\phi-H_n*_{\Y}\phi_i^{(n)}\big{\|}_{2,y}<\frac{\varepsilon}{3}\quad \forall \phi\in\B_r^{2,\infty}.
$$
Then by
\begin{equation*}\begin{split}
\big{\|}H*_{\Y}\phi-H*_{\Y}\phi_i^{(n)}\big{\|}_{2,y}&\le\|H*_{\Y}\phi-H_n*_{\Y}\phi\|_{2,y}+\big{\|}H_n*_{\Y}\phi-H_n*_{\Y}\phi_i^{(n)}\big{\|}_{2,y}
+\big{\|}H_n*_{\Y}\phi_i^{(n)}-H*_{\Y}\phi_i^{(n)}\big{\|}_{2,y}\\
&\le2r\|H-H_n\|_{2,\mu\otimes_{\Y}\mu}+\big{\|}H_n*_{\Y}\phi-H_n*_{\Y}\phi_i^{(n)}\big{\|}_{2,y}\\
\end{split}\end{equation*}
we see that as $n$ sufficiently large,
$$
\min_{1\le i\le k_n}\big{\|}H*_{\Y}\phi-H*_{\Y}\phi_i^{(n)}\big{\|}_{2,y}<\varepsilon\quad \forall \phi\in\B_r^{2,\infty}.
$$
This implies that $H*_{\Y}\B_r^{2,\infty}$ is precompact in $(\mathfrak{L}^2(X,\mathscr{X}^\prime,\mu),\|\cdot\|_{2,y})$, since $\varepsilon$ is arbitrary.
\end{proof}
The proof of Theorem~\ref{thm1.2} is therefore completed.
\end{proof}

Notice here that $\phi\mapsto H*_{\Y}\phi$ is not a compact operator from $(\mathfrak{L}^2(X,\mathscr{X}^\prime,\mu),\|\cdot\|_\infty)$ to $(\mathfrak{L}^2(X,\mathscr{X}^\prime,\mu),\|\cdot\|_2)$.

\subsection{Relatively compact extensions}\label{sec1.4}

The following concept is due to H.~Furstenberg and Y.~Katznelson 1978~\cite{FK}.

\begin{defn}[\cite{FK,Fur,EW}]\label{def1.3}
Given any short factors series $\X\xrightarrow[]{\textit{Id}_X}\X^\prime=(X,\mathscr{X}^\prime,\mu,G)\xrightarrow[]{\pi}\Y$ and any subgroup $\varGamma$ of $G$, for $\pi\colon\X^\prime\rightarrow\Y$ a function $\varphi$ in $\mathfrak{L}^2(X,\mathscr{X}^\prime,\mu)$ is called
\begin{itemize}
\item \textit{FK almost periodic for $\varGamma$} (\textit{FK a.p. for $\varGamma$} for short) if for any $\varepsilon>0$ there exists a finite set of functions $\phi_1,\dotsc,\phi_k$ in $\mathfrak{L}^2(X,\mathscr{X},\mu)$ such that for every $g\in \varGamma$,
\begin{gather*}
{\min}_{1\le j\le k}\|U_g\varphi-\phi_j\|_{2,y}<\varepsilon\quad (\nu\textit{-a.e. }y\in Y).
\end{gather*}
\end{itemize}
We write $L_{\textit{FKap}}^2(X,\mathscr{X}^\prime,\mu,G\textrm{:}\varGamma)$ for the set of all the functions \textit{FK a.p. for $\varGamma$}.

Then $\pi\colon\X^\prime\rightarrow\Y$ is said to be
\begin{itemize}
\item \textit{relatively compact for $\varGamma$} if $L_{\textit{FKap}}^2(X,\mathscr{X}^\prime,\mu,G\textrm{:}\varGamma)$ is dense in $\big{(}\mathfrak{L}^2(X,\mathscr{X}^\prime,\mu),\|\cdot\|_{2,\mu}\big{)}$.
\end{itemize}
Note here that \textit{FK a.p. for $G$} is simply called AP in \cite[Def.~7.18]{EW}. Moreover, it is not required that $\phi_1,\dotsc,\phi_k$ belong to the $\varGamma$-orbit of $\varphi$. Here \textit{FK} is the abbreviation of Furstenberg-Katznelson.
\end{defn}

According to this definition, we easily see that:
\begin{itemize}
\item The relative compactness of $\pi\colon\X^\prime\rightarrow\Y$ appears to depend heavily upon the $G$-factor $(\Y,\pi)$ and also the standard Borel structure of $(X,\mathscr{X})$; this is because the disintegration $\{\mu_y\}_{y\in Y}$ of $\mu$ is closely related with the given condition $(\Y,\pi)$ and the standard Borel structure of $(X,\mathscr{X})$.
\item To get around some hard points appeared in the literature like \cite{FK,Fur,EW}, here we only require loosely that
$\phi_1,\dotsc,\phi_k\in\mathfrak{L}^2(X,\mathscr{X},\mu)$ instead of $\phi_1,\dotsc,\phi_k\in\mathfrak{L}^2(X,\mathscr{X}^\prime,\mu)$.
This point is different with the literature.
\end{itemize}
Clearly, if $\varphi\in\mathfrak{L}^2(X,\mathscr{X}^\prime,\mu)$ is \textit{FK a.p. for $\varGamma$}, then $\varGamma[\varphi]$ is precompact in $(\mathfrak{L}^2(X,\mathscr{X}^\prime,\mu_y),\|\cdot\|_{2,y})$ for $\nu$-\textit{a.e.} $y\in Y$. However, Theorem~\ref{thm1.2} together with Theorem~\ref{thm2.4} below shows that the converse is never true.

Then the following result is an important technique lemma for our arguments later.

\begin{Lem}\label{lem1.4}
Let $\X\xrightarrow[]{\textit{Id}_X}\X^\prime=(X,\mathscr{X}^\prime,\mu,G)\xrightarrow[]{\pi}\Y$ be any short factors series. Given any $R$-submodule $G^\prime$ of the \textit{lcsc} $R$-module $G$, let $\varGamma$ be a countable dense subgroup of $G^\prime$, $\delta\ge0, \epsilon>0$, and $\varphi\in\mathscr{H}^\prime,\phi_1,\dotsc,\phi_k\in\mathscr{H}$ with the property that for any $\gamma\in\varGamma$,
$$
\min_{1\le i\le k}\|U_\gamma\varphi-\phi_i\|_{2,y}<\epsilon\quad \textrm{but for a set of }y\in Y\textrm{ of }\nu\textrm{-measure }\le\delta.
$$
Then for any $g\in G^\prime$,
$$
\min_{1\le i\le k}\|U_g\varphi-\phi_i\|_{2,y}<4\epsilon\quad \textrm{but for a set of }y\in Y\textrm{ of }\nu\textrm{-measure }\le4\delta.
$$
Hence,
\begin{gather*}
L_{\textit{FKap}}^2(X,\mathscr{X}^\prime,\mu,G\mathrm{:}G^\prime)=L_{\textit{FKap}}^2(X,\mathscr{X}^\prime,\mu,G\mathrm{:}\varGamma).
\end{gather*}
\end{Lem}

\begin{proof}
Without loss of generality, we may suppose that $(X,\mathscr{X})$ is a Polish $G$-space by Varadarajan's theorem. Let $g\in G^\prime$ be any given and take any sequence in $\varGamma$, $g_n\to g$. Now select a sequence of bounded $\mathscr{X}^\prime$-functions $\psi_\ell\in\mathscr{H}^\prime$ with $\psi_\ell\to\varphi$ in $\mathscr{H}^\prime$. Since
\begin{gather*}
\|\varphi-\psi_\ell\|_{2,\mu}=\left(\int_Y\|\varphi-\psi_\ell\|_{2,y}^2\nu(dy)\right)^{1/2}\to0\quad \textit{as }\ell\to\infty,
\end{gather*}
hence by passing to a subsequence if necessary, we may assume
\begin{gather*}
\|\varphi-\psi_\ell\|_{2,y}\to0\quad \textit{as }\ell\to\infty\quad (\nu\textit{-a.e. }y\in Y).
\end{gather*}
Inasmuch as $\varGamma$ is countable, so we can choose some $Y_\infty\in\mathscr{Y}$ with $\nu(Y_\infty)=1$ such that
\begin{equation*}
\gamma(Y_\infty)=Y_\infty\ \forall \gamma\in\varGamma\cup\{g\}\quad\textit{and}\quad \lim_{\ell\to\infty}\|\varphi-\psi_\ell\|_{2,y}=0\ \forall y\in Y_\infty.
\end{equation*}
From the fact that for $\nu$-\textit{a.e.}~$y\in Y_\infty$
\begin{equation*}\begin{split}
\min_{1\le i\le k}\|U_g\varphi-\phi_i\|_{2,y}&\le\min_{1\le i\le k}\left(\|U_g\varphi-U_{g_n}\varphi\|_{2,y}+\|U_{g_n}\varphi-\phi_i\|_{2,y}\right)\\
&\le\min_{1\le i\le k}\|U_{g_n}\varphi-\phi_i\|_{2,y}+\|U_g\varphi-U_g\psi_\ell\|_{2,y}+\|U_g\psi_\ell-U_{g_n}\psi_\ell\|_{2,y}+\|U_{g_n}\psi_\ell-U_{g_n}\varphi\|_{2,y}\\
&\le\min_{1\le i\le k}\|U_{g_n}\varphi-\phi_i\|_{2,y}+\|\varphi-\psi_\ell\|_{2,g(y)}+\|U_g\psi_\ell-U_{g_n}\psi_\ell\|_{2,y}+\|\psi_\ell-\varphi\|_{2,g_n(y)}
\end{split}\end{equation*}
together with using Egoroff's theorem for the sequence of functions $\{\|\psi_\ell-\varphi\|_{2,y}\}_{\ell=1}^\infty$ of the variable $y$ on $(Y,\mathscr{Y},\nu)$, it follows that for any $\delta^\prime>3\delta$ we can find a set $U\in\mathscr{Y}$ with $U\subset Y_\infty, \nu(U)>1-\delta^\prime$ such that
\begin{gather*}
\min_{1\le i\le k}\|U_g\varphi-\phi_i\|_{2,y}<4\epsilon\quad\forall y\in U.
\end{gather*}
From this we can conclude the lemma for $\delta^\prime>3\delta$ is arbitrary.
\end{proof}

It should be noted here that the subgroup $\varGamma$ in Lemma~\ref{lem1.4} is not necessarily to be an $R$-submodule of $G$. For example, for $G=\mathbb{R}^r$, the rational subgroup $\mathbb{Q}^r$ is countable dense in $\mathbb{R}^r$ but it is not an $\mathbb{R}$-submodule; $\mathbb{Z}$ is dense in the $p$-adic integers ring $\mathbb{Z}_p$.

\section{Characterizations of relatively compact extension for amenable groups}\label{sec2}
In this section, we will generalize Furstenberg and Katznelson's characterization theorem of relatively compact extensions to any amenable groups following the framework of \cite{Z76, FK,Fur} combining with Varadarajan's isomorphism theorem and the $\mathfrak{L}^2$-mean ergodic theorem.

Throughout this section, let $\X\xrightarrow[]{\textit{Id}_X}\X^\prime=(X,\mathscr{X}^\prime,\mu,G)\xrightarrow[]{\pi}\Y$ be any short factors series under the Borel actions by a same \textit{lcsc} $R$-module $G$ over a ring $(R,+,\cdot)$ not necessarily syndetic, where
\begin{itemize}
\item$(X,\mathscr{X},\mu)$ is a standard Borel probability space so we can disintegrate $\mu$ over $\Y$, $\mu=\int_Y\mu_yd\nu(y)$.
\end{itemize}
As in Remark of $\S\ref{sec1.2}$, write
\begin{gather*}
\mathscr{H}=\mathfrak{L}^2(X,\mathscr{X},\mu)\intertext{and}
\mathscr{H}^\prime=\mathfrak{L}^2(X,\mathscr{X}^\prime,\mu),\quad \mathscr{H}^\prime\otimes_{\Y}\mathscr{H}^\prime=\mathfrak{L}^2(X\times X,\mathscr{X}^\prime\otimes\mathscr{X}^\prime,\mu\otimes_{\Y}\mu),
\end{gather*}
where $\mu\otimes_{\Y}\mu$ is defined based on $\pi\colon\X\rightarrow\Y$ not on $\pi\colon\X^\prime\rightarrow\Y$.

We shall work on the intermediate factor $\X^\prime$ but not on $\X$, as in $\S\ref{sec1}$.

\subsection{Compactness of an extension}
Following the ideas of \cite[Def.~2.1]{FK} and \cite[Sect.~6.3]{Fur}, we now equivalently describe the \textit{relative compactness} of ${\pi}\colon\X^\prime\rightarrow\Y$ respecting any given $R$-submodule $\varGamma$ of $G$.
\begin{itemize}
\item[$\mathbf{C}_1\mathrm{:}$] The set of functions $H*_{\Y}\phi$ spans a dense subset of $\mathscr{H}^\prime$ as $H(x,x^\prime)$ ranges over the bounded $\varGamma$-invariant functions in $\mathscr{H}^\prime\otimes_{\Y}\mathscr{H}^\prime$ and $\phi(x^\prime)$ ranges over the bounded functions in $\mathscr{H}^\prime$.

\item[$\mathbf{C}_2\mathrm{:}$] $\overline{L_{\textit{FKap}}^2(X,\mathscr{X}^\prime,\mu,G\textrm{:}\varGamma)}=\mathscr{H}^\prime$, where a function of \textit{FK a.p. for~$\varGamma$} is defined as in Def.~\ref{def1.3} in $\S\ref{sec1.4}$ following the idea of Furstenberg and Katznelson~\cite{FK}.

\item[$\mathbf{C}_3\mathrm{:}$] Each $\phi\in\mathscr{H}^\prime$ is such that for any $\delta>0$ there exists a $\mathscr{Y}$-set $B\subseteq Y$ with $\nu(B)>1-\delta$ such that the modification
\begin{gather*}
\phi_B(x)=1_{\pi^{-1}[B]}(x)\phi(x)\quad \forall x\in X
\end{gather*}
is \textit{FK a.p. for~$\varGamma$}; that is, for any $\varepsilon>0$ there exists a finite set of functions $\phi_1,\dotsc,\phi_k\in\mathscr{H}$ verifying that for each $g\in\varGamma$,
    \begin{gather*}
    \min_{1\le i\le k}\|U_g\phi_B-\phi_i\|_{2,y}<\varepsilon\quad (\nu\textit{-a.e. }y\in Y).
    \end{gather*}
It is different with $\mathbf{C}_3$ of \cite{FK, Fur} there the choice of $B\in\mathscr{Y}$ depends not only on $\delta$ but also on $\varepsilon$. This $\mathbf{C}_3$-condition implies \cite[Lemma~7.24]{EW} for $R=\mathbb{Z}$.

\item[$\mathbf{C}_4\mathrm{:}$] Any $\phi\in\mathscr{H}^\prime$ is \textit{FK $\nu$-almost almost-periodic for $\varGamma$} (\textit{FK a.a.p. for~$\varGamma$} for short) in the sense that for any $\delta>0$ and $\varepsilon>0$, there exists a finite set of functions $\phi_1,\dotsc,\phi_k$ in $\mathscr{H}$ such that for each $g\in\varGamma$,
    \begin{gather*}
    \min_{1\le i\le k}\|U_g\phi-\phi_i\|_{2,y}<\varepsilon\quad
    \textit{but for a set of }y\in Y\textrm{ of }\nu\textit{-measure less than }\delta.
    \end{gather*}
    We will write $L_{\textsl{FKaap}}^2(X,\mathscr{X}^\prime,\mu,G\textrm{:}\varGamma)$ for the set of all \textit{FK a.a.p. for~$\varGamma$} $\mathfrak{L}^2$-functions. Here `\textit{FK}' is for Furstenberg-Katznelson.

\item[$\mathbf{C}_5\mathrm{:}$] Let $\{F_n\}_{1}^\infty$ be a weak F{\o}lner sequence in $\varGamma$ with a left Haar measure $m_\varGamma$ or write $dg$.\footnote{Let $G$ be an amenable \textit{lcsc} group with a fixed left Haar measure $m_G$. Then there always exists a so-called \textit{weak F{\o}lner sequence} $\{F_n\}_{1}^\infty$ in $G$, i.e., a sequence of nonull compact subsets $F_n$ of positive Haar measure of $G$ such that
\begin{equation*}
\lim_{n\to\infty}\frac{m_G(gF_n\vartriangle F_n)}{m_G(F_n)}=0\quad\forall g\in G.
\end{equation*}
It is a well-known fact that every \textit{lcsc} abelian group is amenable; see, e.g.,~\cite{Pat,EW}. Thus an \textit{lcsc} $R$-module is amenable by our convention before.} For each $\varphi\in\mathscr{H}^\prime$, we define
   \begin{gather*}
   P(\varphi\otimes\bar{\varphi})=\textit{w-}\lim_{n\to\infty}\frac{1}{m_\varGamma(F_n)}\int_{F_n}U_g(\varphi\otimes\bar{\varphi})dg,
   \end{gather*}
   here the \textit{w-}limit means the convergence in the weak topology of $\mathscr{H}\otimes_{\Y}\mathscr{H}$. Then
   \begin{itemize}
   \item $P(\varphi\otimes\bar{\varphi})=0$ if and only if $\varphi=0$.
   \end{itemize}

\item[$\mathbf{C}_6\mathrm{:}$] Let $\{F_n\}_{1}^\infty$ be a weak F{\o}lner sequence in $\varGamma$ with a left Haar measure $m_\varGamma$ or $dg$. For each $\varphi\in\mathscr{H}^\prime$, set
   \begin{gather*}
   P(\varphi\otimes\bar{\varphi})=\textit{s-}\lim_{n\to\infty}\frac{1}{m_\varGamma(F_n)}\int_{F_n}U_g(\varphi\otimes\bar{\varphi})dg,
   \end{gather*}
   where the \textit{s-}limit exists in the $\mathfrak{L}^2$-norm of $\mathscr{H}\otimes_{\Y}\mathscr{H}$. Then
   \begin{itemize}
   \item $P(\varphi\otimes\bar{\varphi})=0$ if and only if $\varphi=0$.
   \end{itemize}
\end{itemize}
It should be noted that:
\begin{itemize}
\item In $\mathbf{C}_3$ and $\mathbf{C}_4$, we did not require that $\phi_1,\dotsc,\phi_k\in\mathscr{H}^\prime$ instead of $\phi_1,\dotsc,\phi_k\in\mathscr{H}$.
\item By the $\mathfrak{L}^2$-mean ergodic theorem~\cite{Dai-16}, $P(\varphi\otimes\bar{\varphi})$ is well defined in $\mathbf{C}_5$ and $\mathbf{C}_6$.
\end{itemize}

The following lemma is useful for us to understand the \textit{FK} $\nu$-almost almost-periodicity property of the extension $\pi\colon\X^\prime\rightarrow\Y$.

\begin{Lem}\label{lem2.1}
Let $\varGamma$ be any $R$-submodule of $G$. If $\varphi\in\mathscr{H}^\prime$ is precompact for $\varGamma$ in $\mathscr{H}^\prime$, i.e., the $\varGamma$-orbit $\varGamma[\varphi]$ is precompact in $(\mathfrak{L}^2(X,\mathscr{X}^\prime,\mu),\|\cdot\|_2)$, then $\varphi\in L_{\textit{FKaap}}^2(X,\mathscr{X}^\prime,\mu,G\mathrm{:}\varGamma)$.
\end{Lem}

\begin{proof}
Let $\delta>0$ and $\varepsilon>0$ be any given. Since $\varGamma[\varphi]$ is precompact in $(\mathfrak{L}^2(X,\mathscr{X}^\prime,\mu),\|\cdot\|_2)$, there is a finite set of functions $\{\phi_1,\dotsc, \phi_\ell\}\subset\varGamma[\varphi]$ such that
$$
\min_{1\le j\le \ell}\|U_g\varphi-\phi_j\|_{2,\mu}<\frac{1}{2}\varepsilon\sqrt{\delta}\quad \forall g\in\varGamma.
$$
For any $g\in\varGamma$, we set
$$
B=\left\{y\in Y\,\big{|}\,\min_{1\le j\le \ell}\|U_g\varphi-\phi_j\|_{2,y}<\frac{1}{2}\varepsilon\right\}.
$$
Then
\begin{gather*}
\frac{1}{4}\varepsilon^2\delta>\min_{1\le j\le \ell}\int_Y\|U_g\varphi-\phi_j\|_{2,y}^2\nu(dy)\ge\int_Y\min_{1\le j\le \ell}\|U_g \varphi-\phi_j\|_{2,y}^2\nu(dy)\ge\frac{1}{4}\varepsilon^2\nu(Y-B)\intertext{and so}\nu(B)\ge1-\delta.
\end{gather*}
This implies that $\varphi\in L_{\textit{FKaap}}^2(X,\mathscr{X}^\prime,\mu,G\mathrm{:}\varGamma)$.

The proof of Lemma~\ref{lem2.1} is thus completed.
\end{proof}

For our convenience, we now introduce another standard notation, which is completely independent of the factor $\Y$ and the factor $G$-map $\pi$ from $\X^\prime$ to $\Y$.

\begin{defn}[{cf.~\cite{Z76,Fur}}]\label{def2.2}
For the intermediate factor $\X^\prime=(X,\mathscr{X}^\prime,\mu,G)$, any function $\phi\in \mathfrak{L}^2(X,\mathscr{X}^\prime,\mu)$ is called
\begin{itemize}
\item (von Neumannian) \textit{almost periodic for an $R$-submodule $\varGamma$} of $G$ (\textit{a.p. for~$\varGamma$} for short) if the partial orbit $\varGamma[\phi]$ is precompact in the space $\big{(}\mathfrak{L}^2(X,\mathscr{X}^\prime,\mu),\|\cdot\|_2\big{)}$.
\end{itemize}
Write $L_{\textit{ap}}^2(X,\mathscr{X}^\prime,\mu,G\mathrm{:}\varGamma)$ for the set of all the $\mathfrak{L}^2$-functions of \textit{a.p. for~$\varGamma$} on $X$.
\end{defn}

Then by Lemma~\ref{lem2.1}, every \textit{a.p. for $\varGamma$} function is an \textit{FK a.a.p. for $\varGamma$} function. Clearly only from definitions, \textit{a.p.} $\not=$ \textit{FK a.a.p.} for a same submodule $\varGamma$. We will construct an example later; see Corollary~\ref{cor4.6} in $\S\ref{sec4.2}$.

The following lemma is obvious but it is very useful.

\begin{Lem}\label{lem2.3}
Let $G$ be an \textit{lcsc} $R$-module, $\X\xrightarrow[]{\textit{Id}_X}\X^\prime\xrightarrow[]{\pi}\Y$ a short factors series and $\varGamma\subseteq G$ an $R$-submodule. Then the following statements hold:
\begin{enumerate}
\item[$(1)$] $L_{\textit{FKaap}}^2(X,\mathscr{X}^\prime,\mu,G\mathrm{:}\varGamma)$ forms an algebra and it is a closed subspace of $\big{(}\mathfrak{L}^2(X,\mathscr{X}^\prime,\mu),\|\cdot\|_2\big{)}$.
\item[$(2)$] If $\varphi\in\mathscr{H}^\prime$ is \textit{FK a.a.p. for $G$}, then $U_g\varphi$ is also \textit{FK a.a.p. for $G$} for each $g\in G$.
\item[$(3)$] If $H(x,x^\prime)\in\mathscr{H}^\prime\otimes_{\Y}\mathscr{H}^\prime$ is $\varGamma$-invariant and $\phi(x^\prime)\in \mathfrak{L}^\infty(X,\mathscr{X}^\prime,\mu)$, then
    \begin{equation}\label{eq2.1}
    U_g(H*_{\Y}\phi)=H*_{\Y}(U_g\phi)
    \end{equation}
     for every $g\in\varGamma$.
\end{enumerate}
\end{Lem}
\begin{proof}
(1) and (2) follow from $\mathbf{C_4}$; (3) follows at once from (\ref{eq1.1}).
\end{proof}
\subsection{Characterizing compact extensions}\label{sec2.2}
When $G$ is a finitely generated free abelian group, i.e., $G\approx\mathbb{Z}^d$ for some integer $d\ge1$, then it is known that $G$ is amenable and $\mathbf{C}_1\sim\mathbf{C}_5$ are equivalent for $\varGamma=G$; cf.~\cite[Theorem~2.1 and Proposition~2.5]{FK} and also~\cite[Theorem~6.13]{Fur}.
We now will extend Furstenberg and Katznelson's characterization theorem to an \textit{lcsc} module as follows.

\begin{Thm}\label{thm2.4}
Let $\X\xrightarrow[]{\textit{Id}_X}\X^\prime=(X,\mathscr{X}^\prime,\mu,G)\xrightarrow[]{\pi}\Y$ be a short factors series and let $\varGamma$ be an $R$-submodule of the \textit{lcsc} $R$-module $G$. Then $\mathbf{C}_1\sim\mathbf{C}_6$ are equivalent to each other for $\varGamma$.
\end{Thm}

\begin{proof}
We will prove Theorem~\ref{thm2.4} by finely modifying Furstenberg and Katznelson's framework developed in the work \cite{FK}.

$\mathbf{C}_1\Rightarrow\mathbf{C}_2$. Clearly any linear combination of functions of \textit{FK a.p. for~$\varGamma$} in $\mathscr{H}^\prime$ is also an \textit{FK a.p. for~$\varGamma$} function. In order to prove that $\mathbf{C}_1\Rightarrow\mathbf{C}_2$, it will suffice to show that by an arbitrary $\mathfrak{L}^2$-norm small modification of a function of the form $H*_{\Y}\phi$ described as in $\mathbf{C}_1$, we can obtain an \textit{FK a.p. for~$\varGamma$} function in $\mathscr{H}^\prime$. Let $\|\phi\|_\infty>0$; otherwise $H*_{\Y}\phi\equiv0$ and nothing needs to do.

According to Lemma~\ref{lem1.4} with $\delta=0$, there is no loss of generality in assuming that $\varGamma$ is countable only thought of as a subgroup of $G$.
As in Step 2 of the proof of Theorem~\ref{thm1.2}, there exists $H_n\to H$ in $\mathscr{H}^\prime\otimes_{\Y}\mathscr{H}^\prime$ where each $H_n(x,x^\prime)$ has the special form of finite linear combination $\sum_j\psi_j\otimes\psi_j^\prime$ with $\psi_j,\psi_j^\prime$ in $\mathfrak{L}^\infty(X,\mathscr{X}^\prime,\mu)$. That means that $\|H-H_n\|_{2,y}\to0$ as a sequence of functions of $y\in Y$ in $\mathfrak{L}^2(Y,\mathscr{Y},\nu)$. Passing to a subsequence if necessary, we assume without loss of generality that
\begin{equation*}
\|H-H_n\|_{2,y}\to0\ \textit{ as }n\to\infty\quad \nu\textit{-a.e. }y\in Y.
\end{equation*}
For any $\eta>0$, by Egoroff's theorem we can find a $\mathscr{Y}$-subset $E$ of $Y$ with $\nu(E)<\eta$ such that
\begin{gather*}
\|H-H_n\|_{2,y}\to0\ \textit{ as }n\to\infty\quad \textit{uniformly for }y\in Y-E.
\end{gather*}
Since $\varGamma$ is countable, we can define the $\varGamma$-invariant set
\begin{gather}
F={\bigcap}_{g\in\varGamma}g^{-1}[E]=Y-{\bigcup}_{g\in\varGamma}g^{-1}[Y-E]\in\mathscr{Y}\label{eq2.2}
\intertext{and let}
\varphi(x)=\begin{cases}H*_{\Y}\phi(x)& \textrm{if }x\not\in\pi^{-1}[F],\\ 0& \textrm{if }x\in\pi^{-1}[F].\end{cases}\label{eq2.3}
\end{gather}
Since $F\subseteq E$ and $\mu\circ\pi^{-1}=\nu$, $\nu(F)<\eta$ so that $\varphi(x)$ defined by (\ref{eq2.3}) differs from $H*\phi(x)$ only on an $\mathscr{X}^\prime$-set of $\mu$-measure less than $\eta$. It is sufficient to verify that $\varphi(x)$ is \textit{FK a.p. for~$\varGamma$}.

Let $\varepsilon>0$ be an arbitrary constant and let there hold the following assertion.
\begin{Claim}
There exists a finite set of functions $\phi_1,\dotsc,\phi_k\in \mathfrak{L}^2(X,\mathscr{X}^\prime,\mu)$ so that for any $g\in\varGamma$,
\begin{gather*}
\min_{1\le j\le k}\|U_g\varphi-\phi_j\|_{2,y}<\varepsilon
\end{gather*}
for $\nu$-\textit{a.e.} $y\in Y-E$.
\end{Claim}

Then we can find a set of functions $\tilde{\phi}_1,\dotsc,\tilde{\phi}_k$, which is $\varepsilon$-spanning $\varGamma[\varphi]$ in $(\mathfrak{L}^2(X,\mathscr{X}^\prime,\mu_y),\|\cdot\|_{2,y})$ for $\nu$-\textit{a.e.} $y\in Y$. Indeed,
we now express the $\varGamma$-invariant $\mathscr{Y}$-set
\begin{gather*}
B:={\bigcup}_{g\in\varGamma}g^{-1}[Y-E]\quad (=Y-F)
\end{gather*}
as a disjoint union $\sum_nB_n$, where $B_n\subseteq g_n^{-1}[Y-E]$ and $\varGamma=\{g_n; n=1,2,3,\dotsc\}$. We then define $\mathscr{X}$-measurable functions
\begin{equation*}
\tilde{\phi}_j(x)=\begin{cases}
0& \textrm{if }\pi(x)\in F=Y-B,\\ U_{g_n}\phi_j(x) & \textrm{if }x\in\pi^{-1}[B_n],\ n\ge1;
\end{cases}\quad j=1,\dotsc,k.
\end{equation*}
And so we can see that $\{\tilde{\phi}_1,\dotsc,\tilde{\phi}_k\}$ is $\varepsilon$-spanning $\varGamma[\varphi]$ for $\nu$-\textit{a.e.} $y\in Y$, noting that $\mu_y\circ g_n^{-1}=\mu_{g_n(y)}$ and $F,B$ both are $\varGamma$-invariant, and that $\mu_y(F)=1$ for $\nu$-\textit{a.e.} $y\in F$ and $\mu_y(B_n)=1$ for $\nu$-\textit{a.e.} $y\in B_n$.

\medskip
We now proceed to prove the above Assertion.
Since $\|H-H_n\|_{2,y}\to0$ uniformly on $Y-E$, we can find sufficiently large $n$ with $\|H-H_n\|_{2,y}<{\varepsilon}/(2\|\phi\|_\infty)$ for all $y\in Y-E$. According to (\ref{eq1.2}) it follows that for any bounded $\psi\in\mathscr{H}^\prime$ with $\|\psi\|_\infty\le\|\phi\|_\infty$,
\begin{equation}\label{eq2.4}
\|H*_{\Y}\psi-H_n*_{\Y}\psi\|_{2,y}<\frac{\varepsilon}{2\|\phi\|_\infty}\cdot\|\phi\|_\infty=\frac{\varepsilon}{2}\quad \forall y\in Y-E.
\end{equation}
If
\begin{equation*}
H_n(x,x^\prime)=\sum_{j=1}^J\psi_j(x)\psi_j^\prime(x^\prime)\quad \textit{where }\psi_j,\psi_j^\prime\in \mathfrak{L}^\infty(X,\mathscr{X}^\prime,\mu)\textit{ and }1\le J<\infty
\end{equation*}
and $\psi\in\mathscr{H}^\prime$ with $\|\psi\|_\infty\le\|\phi\|_\infty$, then
$$
H_n*_{\Y}\psi=\sum_{j=1}^Jc_j(\psi)\psi_j
$$
where
$$
c_j(\psi)(x)=\int_X\psi_j^\prime(x^\prime)\psi(x^\prime)\mu_{\pi(x)}(dx^\prime)\quad \textit{with }\ |c_j(\psi)(x)|\le\|\psi_j^\prime\|_\infty \|\phi\|_\infty.
$$
Now for this set $\left\{\sum_{j=1}^Jc_j(\psi)\psi_j\,|\,\psi\in\mathscr{H}^\prime, \|\psi\|_\infty\le\|\phi\|_\infty\right\}$, it is obvious that there exists a finite set of functions $\phi_1,\dotsc,\phi_k\in \mathscr{H}^\prime$ so that for any $\psi\in\mathscr{H}^\prime$ with $\|\psi\|_\infty\le\|\phi\|_\infty$ we have
\begin{gather}\label{eq2.5}
\min_{1\le i\le k}\bigg{\|}\sum_{j=1}^Jc_j(\psi)\psi_j-\phi_i\bigg{\|}_{2,y}<\frac{\varepsilon}{2}\quad \forall y\in Y.
\end{gather}
Indeed, let
\begin{gather*}
-\|\phi\|_\infty\max_{1\le j\le J}\|\psi_j^\prime\|_\infty=r_0<r_1<\dotsm<r_l=\|\phi\|_\infty\max_{1\le j\le J}\|\psi_j^\prime\|_\infty
\end{gather*}
such that $r_{i+1}-r_i$ are sufficiently small for $0\le i<l$; then
\begin{gather*}
\phi_i\in\left\{\lambda_1\psi_1+\dotsm+\lambda_J\psi_J\,|\,(\lambda_1,\dotsc,\lambda_J)\in\{r_0,r_1,\dotsc,r_l\}^J\right\}
\end{gather*}
is exactly what desired, noting that $c_j(\psi)(x)$ is a function on $Y$.
Finally combined with (\ref{eq2.1}), (\ref{eq2.4}), and (\ref{eq2.5}) yields the desired result.

$\mathbf{C}_2\Rightarrow\mathbf{C}_3$. Let $\phi\in\mathscr{H}^\prime$ and $\delta>0$. Choose a sequence $\delta_\ell\downarrow0$ as $\ell\to\infty$ with $\sum_\ell\delta_\ell<\delta$. For every $\ell\ge1$, by $\mathbf{C}_2$ there is an $\varphi_\ell\in L_{\textit{FKap}}^2(X,\mathscr{X}^\prime,\mu,G\mathrm{:}\varGamma)$ such that
\begin{gather*}
\|\phi-\varphi_\ell\|_2<\delta_\ell.
\end{gather*}
Let
\begin{gather*}
B_\ell=\left\{y\in Y\,\big{|}\,\|\phi-\varphi_\ell\|_{2,y}\ge\sqrt{\delta_\ell}\right\}\quad \textit{and}\quad B=Y-\bigcup_{\ell=1}^\infty B_\ell.
\end{gather*}
Using the identity
$$
\|\phi-\varphi_\ell\|_2=\left(\int_Y\|\phi-\varphi_\ell\|_{2,y}^2\nu(dy)\right)^{1/2},
$$
we can see that
$$\nu(B_\ell)<\delta_\ell\quad \textit{and so}\quad \nu(B)>1-\delta.$$
Finally, to check the \textit{FK a.p. for~$\varGamma$} property of $f_B$, we fix any $\varepsilon>0$ and choose some $\ell$ with $\sqrt{\delta_\ell}<\frac{1}{2}\varepsilon$. Since $\varphi_\ell\in L_{\textit{FKap}}^2(X,\mathscr{X}^\prime,\mu,G\mathrm{:}\varGamma)$, there exist functions $\phi_1,\dotsc,\phi_{k}\in\mathscr{H}$ such that for any $g\in\varGamma$,
$$
\min_{1\le i\le k}\|U_g\varphi_\ell-\phi_i\|_{2,y}<\frac{\varepsilon}{2}\quad \nu\textit{-a.e. }y\in Y;
$$
and set $\phi_{k+1}\equiv0$. Now let $g\in\varGamma$ be arbitrary. For $\nu$-\textit{a.e.}~$y\in Y$, if $y\in g^{-1}[B]$, then
\begin{equation*}
\|U_g\phi_B-U_g\varphi_\ell\|_{2,y}=\|\phi_B-\varphi_\ell\|_{2,g(y)}=\|\phi-\varphi_\ell\|_{2,g(y)}<\frac{\varepsilon}{2}
\end{equation*}
so that
\begin{subequations}\label{eq2.6}
\begin{align}
{\min}_{1\le i\le k}\|U_g\phi_B-\phi_i\|_{2,y}<\varepsilon;\label{eq2.6a}\\
\intertext{if $y\not\in g^{-1}[B]$, then $\phi_B=0$ as an element of $\mathscr{H}_{g(y)}^\prime$ so that $U_g\phi_B=0$ in $\mathscr{H}_y^\prime$ and further}
\|U_g\phi_B-\phi_{k+1}\|_{2,y}=0.\label{eq2.6b}
\end{align}
\end{subequations}
This verifies the property $\mathbf{C}_3$.

$\mathbf{C}_3\Rightarrow\mathbf{C}_4$. This is obvious by (\ref{eq2.6}) and noting that $\phi_B=\phi$ as an element of $\mathscr{H}_y^\prime$ for $y\in B$.

$\mathbf{C}_4\Rightarrow\mathbf{C}_5$. Let us say that $\phi_1,\dotsc,\phi_k$ is a $(\delta,\varepsilon)$-spanning set for $\phi\in\mathscr{H}^\prime$ if $\mathbf{C}_4$ holds. Without loss of generality we can assume that the $\phi_i$ are bounded functions. Then $\bar{\phi}_i\otimes \phi_i$ is in $\mathfrak{L}^\infty(X\times X,\mathscr{X}\otimes\mathscr{X},\mu\otimes_{\Y}\mu)$. Suppose $P(\phi\otimes\bar{\phi})=0$. Then for each $i=1,\dotsc,k$, \begin{gather*}
\int_{X\times X}\bar{\phi}_i\otimes \phi_i\cdot P(\phi\otimes\bar{\phi})d\mu\otimes_{\Y}\mu=0.
\end{gather*}
Using the definition of $P(\phi\otimes\bar{\phi})$ in $\mathbf{C}_5$ and Fubini's theorem, we find
\begin{gather*}
\lim_{n\to\infty}\frac{1}{m_\varGamma(F_n)}\int_{F_n}\sum_{i=1}^k\int_Y\left|\int_X\bar{\phi}_i(x)\phi(g(x))\mu_y(dx)\right|^2\nu(dy)dg=0
\end{gather*}
It follows that for $n$ sufficiently large there will exist some $g\in F_n$ for which
\begin{gather*}
\sum_{i=1}^k\int_Y\left|\int_X\bar{\phi}_i(x)\phi(g(x))\mu_y(dx)\right|^2\nu(dy)
\end{gather*}
is sufficiently small. Thus, for such $g$, but for a set of $y\in Y$ of $\nu$-measure less than $\delta$ we shall have for all $1\le i\le k$,
\begin{gather*}
\left|\int_X\bar{\phi}_i(x)\phi(g(x))\mu_y(dx)\right|^2<\varepsilon^2.
\end{gather*}
However since $\phi_1,\dotsc,\phi_k$ is a $(\delta,\varepsilon)$-spanning set for $\phi$, $\min_{1\le i\le k}\|U_g\phi-\phi_i\|_{2,y}^2<\varepsilon^2$ but for a set of $y\in Y$ of $\nu$-measure less than $\delta$. It follows that but for a set of $y\in Y$ of $\nu$-measure less than $2\delta$ we have $\|U_g\phi\|_{2,y}^2<3\varepsilon^2$. The same is then true for $\phi$ and since $\delta$ and $\varepsilon$ were arbitrary it follows that $\phi=0$ in $\mathscr{H}^\prime$.

$\mathbf{C}_5\Rightarrow\mathbf{C}_6$. This is immediate by the $\mathfrak{L}^2$-mean ergodic theorem (cf.~\cite{Dai-16}).

$\mathbf{C}_6\Rightarrow\mathbf{C}_1$. In contrast with the statement, we suppose that the set $\mathcal{F}$ of functions, which is defined to be
\begin{gather*}
\mathcal{F}=\big{\{}H*_{\Y}\phi\colon H\in \mathfrak{L}^\infty(X\times X,\mathscr{X}^\prime\otimes\mathscr{X}^\prime,\mu\otimes_{\Y}\mu)\textit{ with }gH=H\; \forall g\in\varGamma, \phi\in \mathfrak{L}^\infty(X,\mathscr{X}^\prime,\mu)\big{\}},
\end{gather*}
were not dense in $\mathscr{H}^\prime$. Then we can choose some $\psi\in \mathfrak{L}^\infty(X,\mathscr{X}^\prime,\mu)$ which is orthogonal to $\mathcal{F}$ in $\mathscr{H}^\prime$ and from the $\mathfrak{L}^2$-mean ergodic theorem~\cite{Dai-16} we form a function in $\mathscr{H}^\prime\otimes_{\Y}\mathscr{H}^\prime$ as follows:
\begin{gather*}
H(x,x^\prime)=\lim_{n\to\infty}\frac{1}{m_\varGamma(F_n)}\int_{F_n}U_g(\psi\otimes\bar{\psi})(x,x^\prime)\,dg=P(\psi\otimes\bar{\psi})(x,x^\prime)
\end{gather*}
over some weak F{\o}lner sequence $\{F_n\}_1^\infty$ in $\varGamma$ with a left Haar measure $m_\varGamma$. The function $H$ is $\varGamma$-invariant and by passing to a subsequence of $\{F_n\}_1^\infty$ we see that $\|H\|_\infty\le\|\psi\otimes\bar{\psi}\|_\infty<\infty$.
So by hypothesis we have $\psi\perp (H*_{\Y}\psi)$. Thus
$$
\int_X\left(\int_XH(x,x^\prime)\psi(x^\prime)\mu_{\pi(x)}(dx^\prime)\right)\bar{\psi}(x)\mu(dx)=0.
$$
By $\mu=\int_Y\mu_y\nu(dy)$ and $\mu_y(\pi^{-1}(y))=1$, we can obtain that
\begin{equation*}\begin{split}
0&=\int_Y\left\{\int_X\left(\int_XH(x,x^\prime)\psi(x^\prime)\mu_{\pi(x)}(dx^\prime)\right)\bar{\psi}(x)\mu_y(dx)\right\}\nu(dy)\\
&=\int_Y\left\{\int_X\int_XH(x,x^\prime)\psi(x^\prime)\bar{\psi}(x)\mu_y(dx^\prime)\mu_y(dx)\right\}\nu(dy)\\
&=\iiint H(x,x^\prime)\cdot\overline{\psi\otimes\bar{\psi}}(x,x^\prime)\mu_y(dx^\prime)\mu_y(dx)\nu(dy)\\
&=\int_{X\times X}H\cdot\overline{\psi\otimes\bar{\psi}}d\mu\otimes_{\Y}\mu.
\end{split}\end{equation*}
Therefore $H\in\mathscr{H}^\prime\otimes_{\Y}\mathscr{H}^\prime$ is orthogonal to $\psi\otimes\bar{\psi}\in\mathscr{H}^\prime\otimes_{\Y}\mathscr{H}^\prime$. Since $H$ is $\varGamma$-invariant, we obtain that $U_g(\psi\otimes\bar{\psi})\perp H$ for all $g\in\varGamma$. From this follows that $P(\psi\otimes\bar{\psi})$ which is the ergodic average of $U_g(\psi\otimes\bar{\psi})$ with respect to $\varGamma$ is orthogonal to $H$. But $P(\psi\otimes\bar{\psi})=H$; so $P(\psi\otimes\bar{\psi})=0$ \textit{a.e.} and finally by $\mathbf{C}_6$ we conclude that $\psi=0$ in $\mathscr{H}^\prime$.

The proof of Theorem~\ref{thm2.4} is thus completed.
\end{proof}

\begin{Lem}\label{lem2.5}
If ${\pi}\colon\X^\prime\rightarrow\Y$ is relatively compact for $R$-submodules $\varGamma^\prime,\varGamma^{\prime\prime}\subset G$, respectively; then it is also relatively compact for $\varGamma^\prime\times\varGamma^{\prime\prime}$.
\end{Lem}

\begin{proof}
This follows from the same argument of \cite[Proposition~2.3]{FK} or \cite[Proposition~6.14]{Fur} by using $\mathbf{C}_4$.
So we omit the details here.
\end{proof}

As to be shown by \cite[Example~7.19]{EW}, it is not true that any $\phi\in \mathfrak{L}^2(X,\mathscr{X}^\prime,\mu)$ is automatically \textit{FK a.p. for~$\varGamma$} for a compact extension $\X$ of $\Y$; however $\mathbf{C}_3$ circumvents this drawback.

Without the amenability, we can obtain the following weak implications:

\begin{Thm}\label{thm2.6}
Let $\varGamma$ be an $R$-submodule of an \textit{lcsc} $R$-module $G$ and $\X\xrightarrow[]{\textit{Id}_X}\X^\prime\xrightarrow[]{\pi}\Y$ a short factors series. Then
$$\mathbf{C}_1\Rightarrow\mathbf{C}_2\Rightarrow\mathbf{C}_3\Rightarrow\mathbf{C}_4$$ for $\pi\colon\X^\prime\rightarrow\Y$.
\end{Thm}

In addition, from Lemma~\ref{lem2.3}-(1) and Assertion in the proof of $\mathbf{C}_1\Rightarrow\mathbf{C}_2$ in Theorem~\ref{thm2.4}, we can easily obtain the following lemma.

\begin{Lem}\label{lem2.7}
Let $\varGamma$ be an $R$-submodule of an \textit{lcsc} $R$-module $G$ and $\X\xrightarrow[]{\textit{Id}_X}\X^\prime\xrightarrow[]{\pi}\Y$ a short factors series. Let $H\in\mathscr{H}^\prime\otimes_{\Y}\mathscr{H}^\prime$ be a $\varGamma$-invariant bounded function and let $\phi\in\mathfrak{L}^\infty(X,\mathscr{X}^\prime,\mu)$; then $H*_{\Y}\phi$ is \textit{FK a.a.p.} for $\varGamma$ as in $\mathbf{C}_4$.
\end{Lem}

The following lemma will be needed in proving our dichotomy theorem in $\S\ref{sec4.1}$.

\begin{Lem}\label{lem2.8}
Let $K$ be a compact subset of the \textit{lcsc} $R$-module $G$, $\varGamma$ an $R$-submodule of $G$, and $\X\xrightarrow[]{\textit{Id}_X}\X^\prime\xrightarrow[]{\pi}\Y$ a short factors series. If
\begin{gather*}
L_{\textit{FKaap}}^2(X,\mathscr{X}^\prime,\mu, G\mathrm{:}\varGamma)=\mathfrak{L}^2(X,\mathscr{X}^\prime,\mu),
\end{gather*}
then each $\varphi\in\mathfrak{L}^2(X,\mathscr{X}^\prime,\mu)$ is \textit{FK a.a.p.} for $K\times\varGamma$.
\end{Lem}

\begin{proof}
In view of Theorem~\ref{thm2.4} (i.e. $\mathbf{C}_2\Leftrightarrow\mathbf{C}_4$) and Lemma~\ref{lem2.3}-(1), it is sufficient to show that if $\varphi\in L_{\textit{FKap}}^2(X,\mathscr{X}^\prime,G\mathrm{:}\varGamma)$ then $\varphi$ is \textit{FK a.a.p. for $K\times\varGamma$}. Now let $\varphi\in L_{\textit{FKap}}^2(X,\mathscr{X}^\prime,G\mathrm{:}\varGamma)$ and $\delta>0,\varepsilon>0$ be any given. Take $\phi_1,\dotsc,\phi_J\in\mathscr{H}$ so that
$$
\min_{1\le j\le J}\|U_g\varphi-\phi_j\|_{2,y}<\frac{\varepsilon}{3}\quad (\nu\textit{-a.e. }y\in Y).
$$
Next we can find some $\epsilon=\epsilon(\delta,\varepsilon,J)>0$ such that
$$
\psi\in\mathscr{H}^\prime,\ \|\psi\|_{2,\mu}<\epsilon\quad\Rightarrow\quad\|\psi\|_{2,y}<\frac{\varepsilon}{3}\quad \textit{but for a set of }y\in Y\textit{ of }\nu\textit{-measure less than }\frac{\delta}{J}.
$$
Since
$$
\Phi=\{U_g\phi_j; g\in K, 1\le j\le J\}
$$
is compact in $\mathscr{H}$, we can find $\frac{\epsilon}{3}$-balls
$$
B_{\epsilon/3}(\psi_1), \dotsc, B_{\epsilon/3}(\psi_K)
$$
in $\mathscr{H}$, which cover $\Phi$. Now for any $f\in K, g\in\varGamma$  we have
\begin{equation*}\begin{split}
\min_{1\le i\le K}\|U_{fg}\varphi-\psi_i\|_{2,y}&\le\min_{1\le i\le K}\left\{\|U_f(U_g\varphi)-U_f\phi_j\|_{2,y}+\|U_f\phi_j-\psi_i\|_{2,y}\right\}\\
&=\min_{1\le i\le K}\left\{\|U_g\varphi-\phi_j\|_{2,f(y)}+\|U_f\phi_j-\psi_i\|_{2,y}\right\}\\
&\le\|U_g\varphi-\phi_j\|_{2,f(y)}+\min_{1\le i\le K}\|U_f\phi_j-\psi_i\|_{2,y}.
\end{split}\end{equation*}
Set
$$
Y_j=\left\{y\in Y\,\big{|}\,\|U_g\varphi-\phi_j\|_{2,f(y)}<\frac{\varepsilon}{3}\right\}.
$$
Clearly, $Y=Y_1\cup Y_2\cup\dotsm\cup Y_J$ ($\nu$-mod 0).
Thus, for $y\in Y_j$,
$$
\min_{1\le i\le K}\|U_{fg}\varphi-\psi_i\|_{2,y}\le\frac{\varepsilon}{3}+\frac{\varepsilon}{3}\quad \textit{but for a set of }y\in Y_j\textit{ of }\nu\textit{-measure less than }\frac{\delta}{J}.
$$
This implies that
$$
\min_{1\le i\le K}\|U_{fg}\varphi-\psi_i\|_{2,y}<\varepsilon\quad \textit{but for a set of }y\in Y\textit{ of }\nu\textit{-measure less than }\delta.
$$
Note that the choice of $\psi_1,\dotsc,\psi_K\in\mathscr{H}$ is independent of $f\in K,g\in\varGamma$.

The proof of Lemma~\ref{lem2.8} is thus completed.
\end{proof}

Finally we note that by the same arguments, we can see that Theorem~\ref{thm2.4} holds for any \textit{lcsc} amenable group $G$ and amenable subgroup $\varGamma$ without the $R$-module structure of $G$.

\section{Alternating theorems for locally compact second countable modules}\label{sec3}
This section is devoted to proving alternating theorems for \textit{lcsc} $R$-modules which generalize Furstenberg and Katznelson's theorem using slightly different methods. We say $\varGamma$ is a \textit{nontrivial} $R$-submodule of an \textit{lcsc} $R$-module $G$ over a ring $(R,+,\cdot)$ if $\varGamma$ is an $R$-submodule with $\varGamma\not=\{I\}$.
\begin{itemize}
\item Let $\X=(X,\mathscr{X},\mu,G)\xrightarrow[]{\textit{Id}_X}\X^\prime=(X,\mathscr{X}^\prime,\mu,G)\xrightarrow[]{\pi}\Y=(Y,\mathscr{Y},\nu,G)$ be a short factors series, where
    \begin{itemize}
    \item $(X,\mathscr{X},\mu)$ is a standard Borel probability space.
    \end{itemize}
\end{itemize}
Differently with \cite{FK,Fur,EW}, we will work on the intermediate factor $\X^\prime$, not directly on $\X$, as in $\S\ref{sec2}$.
\subsection{Totally relatively weak-mixing extensions}
Recall from the viewpoint of group as in \cite{FK, Fur} that $\pi\colon\X^\prime\rightarrow\Y$ is referred to as \textit{relatively weak-mixing for an element $g$} in $G$ if every $g$-invariant (or equivalently $\{g^n\,|\,n\in\mathbb{Z}\}$-invariant) function $H(x,x^\prime)$ in $\mathscr{H}^\prime\otimes_{\Y}\mathscr{H}^\prime$ is a function on $(Y,\mathscr{Y},\nu)$ via the factor $G$-map
$${\pi\times_{\Y}\pi}\colon X\times X\rightarrow Y.$$
Let $\varGamma$ be a subgroup of $G$; then we say $\pi\colon\X^\prime\rightarrow\Y$ is \textit{(totally) relatively weak-mixing for $\varGamma$} if $\pi\colon\X^\prime\rightarrow\Y$ is relatively weak-mixing for every $g, g\not=I$, in $\varGamma$ (cf.~\cite[Def.~6.3]{Fur}).

We now generalize this to modules case as follows:

\begin{defn}\label{def3.1}
Let $\varGamma$ be a nontrivial subset of the \textit{lcsc} $R$-module $G$ and $g\in G, g\not=I$. Then $\pi\colon\X^\prime\rightarrow\Y$ is said to be
\begin{enumerate}
\item[$(a)$] \textit{relatively weak-mixing for $g$} if every $\langle g\rangle_R$-invariant function $H(x,x^\prime)\in\mathscr{H}^\prime\otimes_{\Y}\mathscr{H}^\prime$ is a function on $(Y,\mathscr{Y},\nu)$ via ${\pi\times_{\Y}\pi}\colon X\times X\rightarrow Y$.
\item[$(b)$] \textit{totally relatively weak-mixing for $\varGamma$} if ${\pi}\colon\X^\prime\rightarrow\Y$ is relatively weak-mixing for each $g$ in $\varGamma$ with $g\not=I$ under the sense of $(a)$;

\item[$(c)$] \textit{jointly relatively weak-mixing for $\varGamma$} if every $\langle\varGamma\rangle_R$-invariant function $H(x,x^\prime)\in\mathscr{H}^\prime\otimes_{\Y}\mathscr{H}^\prime$ is a function on $(Y,\mathscr{Y},\nu)$ via ${\pi\times_{\Y}\pi}\colon X\times X\rightarrow Y$. See \cite{Z76,Gla} for $(R,+,\cdot)=(\mathbb{Z},+,\cdot)$ and $\X=\X^\prime$.
\end{enumerate}
\end{defn}

Since every $\langle\varGamma\rangle_R$-invariant function is $\langle g\rangle_R$-invariant for each $g\in\varGamma$, a totally relatively weak-mixing extension must be a jointly relatively weak-mixing extension for $\varGamma$; but the converse does not need to be true.

\begin{remark}
Given any $T\not=I$ in $G$, if $(X,\mathscr{X}^\prime,\mu,T)$ is weakly mixing, then $(X,\mathscr{X}^\prime,\mu,\langle T\rangle_\mathbb{Z})$ is totally weak-mixing (cf.~\cite[Proposition~4.7]{Fur}). However, this is never the case for extensions.
\begin{itemize}
\item If $\pi\colon\X^\prime\rightarrow\Y$ is relatively weak-mixing for $g$ ($g\not=I$), then it is not necessarily totally relatively weak-mixing for $\langle g\rangle_R$ unless $Rr=R$ for each $r\not=0$ like $R$ to be a field.
\end{itemize}
That is to say, ``$\pi\colon\X^\prime\rightarrow\Y$ is relatively weak-mixing for $g\not=I$'' $\not=$ ``$\pi\colon\X^\prime\rightarrow\Y$ is totally relatively weak-mixing for $\langle g\rangle_R$'' in general.
\end{remark}

\begin{prop}
Let $\varGamma$ be a dense subgroup of an \textit{lcsc} group $G$. Then $\pi\colon\X^\prime\rightarrow\Y$ is jointly relatively ergodic for $G$ if and only if so is $\pi\colon(X,\mathscr{X}^\prime,\mu,\varGamma)\rightarrow(Y,\mathscr{Y},\nu,\varGamma)$ for $\varGamma$.
\end{prop}

\begin{proof}
It is sufficient to show that for any $\varphi(x)\in \mathfrak{L}^2(X,\mathscr{X}^\prime,\mu)$, $\varphi(x)$ is (\textit{$\mu$-a.e.}) $G$-invariant if and only if it is (\textit{$\mu$-a.e.}) $\varGamma$-invariant. This is obvious since $g\mapsto U_g\varphi$ is continuous from $G$ to $\mathfrak{L}^2(X,\mathscr{X}^\prime,\mu)$.
\end{proof}

\subsection{Joint case}
Our theorem below is a generalization of Furstenberg and Katznelson's alternating theorem by only considering the $\mathbb{Z}$-module $\mathbb{Z}^r$ and $\varGamma=\langle T\rangle_\mathbb{Z}$ for $T\in G$ with $T\not=I$ (cf.~\cite[Proposition~2.2]{FK}, \cite[Theorem~6.15]{Fur} and \cite[Theorem~7.21]{EW}).
With Theorem~\ref{thm2.4} at hands, now we may independently prove our first alternating theorem as follows:

\begin{Thm}[Alternating Theorem~I]\label{thm3.4}
Let $\varGamma$ be a nontrivial $R$-submodule of an \textit{lcsc} $R$-module $G$, and $\X\xrightarrow[]{\textit{Id}_X}\X^\prime\xrightarrow[]{\pi}\Y$ a short factors series, where $\pi\colon\X^\prime\rightarrow\Y$ is nontrivial. Then either
\begin{enumerate}
\item[$(1)$] $\pi\colon\X^\prime\rightarrow\Y$ is jointly relatively weak-mixing for $\varGamma$;
\end{enumerate}
or
\begin{enumerate}
\item[$(2)$] there exists an intermediate factor $\X^{\prime\prime}=(X,\mathscr{X}^{\prime\prime},\mu,G)$ between $\X^\prime$ and $\Y$ such that:
\begin{itemize}
    \item $\pi\colon\X^{\prime\prime}\rightarrow\Y$ is nontrivial, relatively compact for $\varGamma$;
    \item $\{x\}$ is an $\mathscr{X}^{\prime\prime}$-set for each $x\in X$;
    \item $\textit{Id}_X\colon\X^\prime\rightarrow\X^{\prime\prime}$.
    \end{itemize}
\end{enumerate}
\end{Thm}

\begin{proof}
First of all, note that conditions $\mathbf{C}_1\sim\mathbf{C}_6$ are equivalent to each other for $\varGamma$ under the hypothesis of Theorem~\ref{thm3.4}.

(1): We first note that if $\pi\colon\X^\prime\rightarrow\Y$ is jointly relatively weak-mixing for $\varGamma$, then from the non-triviality of ${\pi}\colon\X^\prime\rightarrow\Y$ it follows that ${\pi}\colon\X^\prime\rightarrow\Y$ is not relatively $\mathbf{C}_1$-compact for $\varGamma$. Indeed, if $\pi\colon\X^\prime\rightarrow\Y$ is relatively compact for $\varGamma$, then for any $\mu\otimes_{\Y}\mu$-\textit{a.e.} $\varGamma$-invariant $H(x,x^\prime)\in\mathscr{H}^\prime\otimes_{\Y}\mathscr{H}^\prime$ and any $\phi\in\mathfrak{L}^\infty(X,\mathscr{X}^\prime,\mu)$, by the jointly relatively weak-mixing,
\begin{gather*}
H*_{\Y}\phi(x)=\int_XH(x,x^\prime)\phi(x^\prime)\mu_{\pi(x)}(dx^\prime)=\int_XH(z,x^\prime)\phi(x^\prime)\mu_{\pi(z)}(dx^\prime)=H*_{\Y}\phi(z)
\end{gather*}
for any $x,z\in\pi^{-1}(y)$ for $\nu$-\textit{a.e.} $y\in Y$; hence by $\mathbf{C}_1$, $\mathscr{H}^\prime=\mathfrak{L}^2(X,\pi^{-1}[\mathscr{Y}],\mu)$ and so $\pi\colon\X^\prime\rightarrow\Y$ is trivial, a contradiction.

Conversely, if $\pi\colon\X^\prime\rightarrow\Y$ is relatively compact for $\varGamma$, then by $\mathbf{C}_1$ it similarly follows that $\pi\colon\X^\prime\rightarrow\Y$ is never jointly relatively weak-mixing for $\varGamma$. So if $(1)$ holds, we then can stop here.

(2): Next we assume $\pi\colon\X^\prime\rightarrow\Y$ is not jointly relatively weak-mixing for $\varGamma$. This means that we can find
a bounded $\varGamma$-invariant ($\mu\otimes_{\Y}\mu$-\textit{a.e.}) function $H\in\mathscr{H}^\prime\otimes_{\Y}\mathscr{H}^\prime$ on $X\times X$, which is not a (lift of some) function on $(Y,\mathscr{Y},\nu)$. Then replacing $H(x,x^\prime)$ by $H(x^\prime,x)$ if necessary, we can conclude that:
\begin{itemize}
\item[] There exists some $\varphi\in \mathfrak{L}^\infty(X,\mathscr{X}^\prime,\mu)$ such that $H*_{\Y}\varphi\colon X\rightarrow\mathbb{C}$ is bounded but not a function on $(Y,\mathscr{Y},\nu)$. See, e.g., \cite[Lemma~7.22]{EW}.
\end{itemize}
By Lemma~\ref{lem2.7}, $H*_{\Y}\varphi$ is a bounded function of \textit{FK a.a.p. for~$\varGamma$} as in $\mathbf{C}_4$. Hence there always exist bounded functions in $L_{\textit{FKaap}}^2(X,\mathscr{X}^\prime,\mu,G\mathrm{:}\varGamma)$ which are not functions on $(Y,\mathscr{Y},\nu)$.

By Lemma~\ref{lem2.3}-(1), it is clear that sums, products and limits of $\mathfrak{L}^2$-functions of \textit{FK a.a.p. for~$\varGamma$} are still $\mathfrak{L}^2$-functions of \textit{FK a.a.p. for~$\varGamma$}.

Moreover (the lifting of) all of the functions in $\mathfrak{L}^\infty(Y,\mathscr{Y},\nu)$ are \textit{FK a.a.p. for~$\varGamma$} for the extension $\pi\colon\X^\prime\rightarrow\Y$. Indeed, for any $A\in\mathscr{Y}$, let $$H(x,x^\prime)=1_{\pi^{-1}\left[\cup_{g\in\varGamma}g^{-1}[A]\right]}(x)\in\mathscr{H}^\prime\otimes_{\Y}\mathscr{H}^\prime$$
which is $\varGamma$-invariant on $X\times X$ and we set $\phi(x^\prime)=1_{\pi^{-1}[A]}(x^\prime)$; then it is easy to check that
$1_{\pi^{-1}[A]}(x)=H*_{\Y}\phi(x)$. This shows $1_{\pi^{-1}[A]}(\centerdot)$ is \textit{FK a.a.p. for~$\varGamma$} by Lemma~\ref{lem2.7}. Then $\mathfrak{L}^\infty(Y,\mathscr{Y},\nu)\subseteq L_{\textit{FKaap}}^2(X,\mathscr{X}^\prime,\mu,G\mathrm{:}\varGamma)$ by Lemma~\ref{lem2.3}.

In addition, for any point $x_0\in X$, $1_{\{x_0\}}$ is \textit{FK a.a.p. for~$\varGamma$}. Indeed, if $\mu(\{x_0\})=0$, the assertion holds for $1_{\{x_0\}}(x)\equiv0$ \textit{a.e.}; otherwise, $x_0$ is a periodic point for $G$and then the assertion also holds automatically.

Let
$$\mathscr{X}^{\prime\prime}=\big{\{}A\in\mathscr{X}^\prime\,|\,1_A(\centerdot)\textrm{ is \textit{FK a.a.p. for~$\varGamma$}}\textrm{ for }\X^\prime\xrightarrow{\pi}\Y\big{\}},$$
which contains $\pi^{-1}[\mathscr{Y}]$ because $1_{\pi^{-1}[A]}$ is in $L_{\textit{FKaap}}^2(X,\mathscr{X}^\prime,\mu,G\mathrm{:}\varGamma)$ for each $A\in\mathscr{Y}$; then it is a $\sigma$-subalgebra of $\mathscr{X}^\prime$ by the Halmos Monotone-class Theorem, since $L_{\textit{FKaap}}^2(X,\mathscr{X}^\prime,\mu,G\mathrm{:}\varGamma)$ is an algebra and it is closed.
In addition $\mathscr{X}^{\prime\prime}$ is obviously $\varGamma$-invariant and $\{x\}\in\mathscr{X}^{\prime\prime}$ for all $x\in X$; moreover,
\begin{itemize}
\item $\mathscr{X}^{\prime\prime}$ is $G$-invariant by the commutativity of $G$.
\end{itemize}
This shows that if we define $\X^{\prime\prime}=(X,\mathscr{X}^{\prime\prime},\mu,G)$, where the $G$-action map is defined by $(g,x)=g(x)$ same as in $\X^\prime$; then
$\X^{\prime\prime}$ is a factor of $\X^\prime$ via the factor $G$-map $\textit{Id}_X\colon X\rightarrow X$ and $\Y$ is a factor of $\X^{\prime\prime}$ via the factor $G$-map $\pi\colon X\rightarrow Y$.

We now claim that $\mathfrak{L}^2(X,\mathscr{X}^{\prime\prime},\mu)=L_\textit{FKaap}^2(X,\mathscr{X}^{\prime\prime},\mu,G\mathrm{:}\varGamma)$. Indeed, let $\phi$ be $\mathscr{X}^{\prime\prime}$-measurable, then $\phi$ is the limit in
$\mathfrak{L}^2(X,\mathscr{X}^{\prime\prime},\mu)$ of finite linear combinations of characteristic functions of $\mathscr{X}^{\prime\prime}$-sets, and hence $\phi$ belongs to $L_\textit{FKaap}^2(X,\mathscr{X}^{\prime\prime},\mu,G\mathrm{:}\varGamma)$, since $L_\textit{FKaap}^2(X,\mathscr{X}^{\prime\prime},\mu,G\mathrm{:}\varGamma)$ is closed.\footnote{It should be noted here that if $\phi\in\mathfrak{L}^2(X,\mathscr{X}^{\prime\prime},\mu)\cap L_{\textit{FKaap}}^2(X,\mathscr{X}^\prime,\mu,G\mathrm{:}\varGamma)$, then $\phi\in L_{\textit{FKaap}}^2(X,\mathscr{X}^{\prime\prime},\mu,G\mathrm{:}\varGamma)$. This is just the reason why we did not require $\phi_1,\dotsc,\phi_k\in\mathscr{H}^\prime$ instead of $\phi_1,\dotsc,\phi_k\in\mathscr{H}$
in $\mathbf{C}_4$. Otherwise, we cannot check that $1_A$ is \textit{FK a.a.p.} for $\varGamma$ (for needing to find $\phi_1,\dotsc,\phi_k\in\mathscr{H}^{\prime\prime}$) for any $A\in\mathscr{X}^{\prime\prime}$.
}
Therefore we can conclude that $\pi\colon\X^{\prime\prime}\rightarrow\Y$ is a relatively compact extension for $\varGamma$.

Finally we notice that ${\pi}^{-1}[\mathscr{Y}]\subsetneq\mathscr{X}^{\prime\prime}$ ($\mu$-mod $0$); this is because
\begin{gather*}
L_\textit{FKaap}^2(X,\mathscr{X}^\prime,\mu,G\mathrm{:}\varGamma)\subseteq L_\textit{FKaap}^2(X,\mathscr{X}^{\prime\prime},\mu,G\mathrm{:}\varGamma)
\end{gather*}
and there always exist bounded functions in $L_{\textit{FKaap}}^2(X,\mathscr{X}^\prime,\mu,G\mathrm{:}\varGamma)$ which are not functions on $(Y,\mathscr{Y},\nu)$.

Therefore $\pi\colon\X^{\prime\prime}\rightarrow\Y$ is a nontrivial relatively compact extension of $\Y$ for $\varGamma$.
The proof of Theorem~\ref{thm3.4} is thus completed.
\end{proof}

Using Theorem~\ref{thm2.6} and Lemma~\ref{lem2.3}, we can similarly obtain the following

\begin{Thm}[Alternating Theorem~II]\label{thm3.5}
Let $G$ be an \textit{lcsc} amenable group and $\X\xrightarrow[]{\textit{Id}_X}\X^\prime\xrightarrow[]{\pi}\Y$ a short factors series, where $\pi\colon\X^\prime\rightarrow\Y$ is nontrivial. Then at least one of the following two statements holds:
\begin{enumerate}
\item[$(1)$] $\pi\colon\X^\prime\rightarrow\Y$ is jointly relatively weak-mixing for $G$;

\item[$(2)$] there exists an intermediate factor $\X^{\prime\prime}=(X,\mathscr{X}^{\prime\prime},\mu,G)$ between $\X^\prime$ and $\Y$ such that $\pi\colon\X^{\prime\prime}\rightarrow\Y$ is nontrivial relatively compact for $G$.
\end{enumerate}
\end{Thm}

\begin{proof}
If (1) holds, we then stop here. Now assume $\pi\colon\X^\prime\rightarrow\Y$ is not jointly relatively weak-mixing for $G$. The proof is almost same as that of the second part of Theorem~\ref{thm3.4}. We only need to note that the $G$-invariance of $\mathscr{X}^{\prime\prime}$ here follows from Lemma~\ref{lem2.3}-$(3)$.
\end{proof}

Note that we can further disintegrate $\mu$ over the factor $\textit{Id}_X\colon\X\rightarrow\X^{\prime\prime}$.

\subsection{Individual case}
A special case of Theorem~\ref{thm3.4}  is the following

\begin{Thm}[Alternating Theorem~III]\label{thm3.6}
Let $G$ be an \textit{lcsc} $R$-module and $\X\xrightarrow[]{\textit{Id}_X}\X^\prime\xrightarrow[]{\pi}\Y$ a short factors series where $\pi\colon\X^\prime\rightarrow\Y$ is nontrivial. Then at least one of the followings holds:
\begin{enumerate}
\item[$(1)$] $\pi\colon\X^\prime\rightarrow\Y$ is totally relatively weak-mixing for $G$;
\item[$(2)$] there exists an intermediate factor $\X^{\prime\prime}=(X,\mathscr{X}^{\prime\prime},\mu,G)$ between $\X^\prime$ and $\Y$ with $\mathscr{X}^{\prime\prime}\subset\mathscr{X}^\prime$
    such that
    \begin{itemize}
    \item $\pi\colon\X^{\prime\prime}\rightarrow\Y$ is a nontrivial, relatively compact extension for $\langle g\rangle_R$ for some element $g$ in $G$ with $g\not=I$;
        \item $\{x\}$ is an $\mathscr{X}^{\prime\prime}$-set for each $x\in X$;
        \item $\textit{Id}_X\colon\X^\prime\rightarrow\X^{\prime\prime}$.
        \end{itemize}
\end{enumerate}
\end{Thm}

\begin{proof}
If (1) does not hold, then there exists some element $g\in G, g\not=I$, such that ${\pi}\colon\X^\prime\rightarrow\Y$ is not jointly relatively weak-mixing for the $R$-submodule $\langle g\rangle_R$. Then by Theorem~\ref{thm3.4}, the statement (2) of Theorem~\ref{thm3.6} holds. The proof of Theorem~\ref{thm3.6} is thus complete.
\end{proof}

\subsection{A remark on intermediate factors}\label{sec3.3}
It should be noted here that although the intermediate factors $\X^{\prime\prime}=(X,\mathscr{X}^{\prime\prime},\mu,G)$ in Theorems~\ref{thm3.4}, \ref{thm3.5} and \ref{thm3.6} are not necessarily to be standard Borel $G$-spaces (i.e., there does not need to exist a topology on $X$ so that $(X,\mathscr{X}^{\prime\prime},\mu)$ is a Polish Borel probability space and the $G$-action is Borel), yet Theorem~\ref{thm2.4} is still valid for these factors of the standard Borel system $\X$.
This point is important for us to build up a structure theorem that consists of ``primitive'' links in the factors chain in $\S\ref{sec4}$ below.

\section{Dichotomy theorem and Furstenberg structure theorem}\label{sec4}
We will generalize Furstenberg's dichotomy and structure theorems in this section. First we suppose that $G$ is an \textit{lcsc} $R$-module over a ring $(R,+,\cdot)$ as in Introduction and assume $${\pi}\colon\X=(X,\mathscr{X},\mu,G)\rightarrow\Y=(Y,\mathscr{Y},\nu,G)$$
is an extension under Borel acting by $G$, where
\begin{itemize}
\item $(X,\mathscr{X},\mu)$ is a standard Borel $G$-space and $(Y,\mathscr{Y},\nu)$ is a Borel $G$-space.
\end{itemize}
Then conditions $\mathbf{C}_1\sim\mathbf{C}_6$ are equivalent to each other by Theorem~\ref{thm2.4}.

The main arguments of this section will need the condition that $(R,+,\cdot)$ is a syndetic ring and $G$ is Noetherian.

\subsection{Primitive factors}\label{sec4.0}
Let $\X\xrightarrow[]{\textit{Id}_X}\X^\prime\xrightarrow[]{\pi}\Y$ be a short factors series, where $\pi\colon\X^\prime\rightarrow\Y$ is nontrivial and $\textit{Id}_X$ is the identity of $X$ to itself.

\begin{defn}[{cf.~\cite[Def.~6.5]{Fur} and \cite{FK} for $R=\mathbb{Z}$}]\label{def4.1}
We shall say that $\pi\colon\X^\prime\rightarrow\Y$ is \textit{primitive} if one of the following holds:
\begin{enumerate}
\item[$(a)$] $\pi\colon\X^\prime\rightarrow\Y$ is totally relatively weak-mixing for $G$;
\item[$(b)$] $\pi\colon\X^\prime\rightarrow\Y$ is relatively compact for $G$;
\item[$(c)$] $G$ is the nontrivial direct product of two $R$-submodules $G=G_{rc}\times G_{rw}$ such that $\pi\colon\X^\prime\rightarrow\Y$ is relatively compact for $G_{rc}$ and totally relatively weak-mixing for $G_{rw}$.
\end{enumerate}
The case $(c)$ will be called ``chaotically primitive''.
\end{defn}

We notice here that $(a)$ and $(c)$ both require that $(X,\mathscr{X},\mu)$ be a standard Borel $G$-space; otherwise, there is no the definition of relatively weak-mixing for $G$ or $G_{{rw}}$ over $\X^\prime$.

\subsection{Dichotomy theorem}\label{sec4.1}
Combining Theorem~\ref{thm3.6} and Lemma~\ref{lem2.5} we can obtain our dichotomy theorem--Theorem~\ref{thm4.2}, which is a generalization of the Furstenberg-Katznelson dichotomy theorem (cf.~\cite[Theorem~2.4]{FK} and also \cite[Theorem~6.16]{Fur}) and of Furstenberg's \cite[Theorem~10.3]{F63}.

\begin{Thm}\label{thm4.2}
Let $G$ be an \textit{lcsc} \textbf{\textit{Noetherian}} $R$-module over a \textbf{\textit{syndetic}} ring $(R,+,\cdot)$. If $\pi\colon\X\rightarrow\Y$ is a nontrivial extension, then one can find a short factors series, $\X\xrightarrow{\textit{Id}_X}\X^\prime=(X,\mathscr{X}^\prime,\mu,G)\xrightarrow{\pi}\Y$, such that
\begin{itemize}
\item $\pi\colon\X^\prime\rightarrow\Y$ is a nontrivial primitive extension of $\Y$.
\end{itemize}
\end{Thm}

\begin{proof}
Let $\pi\colon\X\rightarrow\Y$ itself be neither totally relatively weak-mixing for $G$ nor relatively compact for $G$; otherwise the statement holds trivially by taking $\X^\prime=\X$.

Then by Theorem~\ref{thm3.6}, it follows that there is a factor $\X^\prime=(X,\mathscr{X}^\prime,\mu,G)$ of $\X$ such that $\pi\colon\X^\prime\rightarrow\Y$ is a nontrivial relatively compact extension for $\langle g\rangle_R$, for some $g\in G$ with $g\not=I$.
Let $\varGamma_{rc}$ be the nonempty collection of all $R$-submodules of $G$ for which $\pi\colon\X^\prime\rightarrow\Y$ is a nontrivial relatively compact extension for some factor $(\X^\prime,\pi)$ of $\X$. Since $G$ is Noetherian, we can choose a maximal $R$-submodule from $\varGamma_{rc}$, say $G_{rc}$, such that $\pi\colon\X^\prime\rightarrow\Y$ is a nontrivial relatively compact extension for $G_{rc}$ and that $G_{rc}\not=\{I\}$, where $\X^\prime=(X,\mathscr{X}^\prime,\mu,G)$ is such that $\{x\}\in\mathscr{X}^\prime$ for each $x\in X$.

Let $G_{rc}\not=G$; otherwise Theorem~\ref{thm4.2} holds. We then can claim that for any $g\in G\setminus G_{rc}$, $\X^\prime\xrightarrow{\pi}\Y$ is not relatively compact for $\langle g\rangle_R$. For if not, one can find some $g\in G\setminus G_{rc}$ so that ${\pi}\colon\X^{\prime}\rightarrow\Y$ is a nontrivial compact extension for the $R$-submodule $\langle g\rangle_R\times G_{rc}$ that is larger than $G_{rc}$, by Lemma~\ref{lem2.5}. This contradicts the maximality of $G_{rc}$.

Now we will consider the short factors series: $\X\xrightarrow{\textit{Id}_X}\X^\prime\xrightarrow{\pi}\Y$.
We will assert that for any $g\in G\setminus G_{rc}$, $\pi\colon\X^\prime\rightarrow\Y$ is \textit{jointly} relatively weak-mixing for the $R$-submodule $\langle g\rangle_R$.
Indeed, if not, then by Theorem~\ref{thm3.4} it follows that there is a factor of $\X^\prime$, $\X^{\prime\prime}=(X,\mathscr{X}^{\prime\prime},\mu, G)$, such that $\pi\colon\X^{\prime\prime}\rightarrow\Y$ is nontrivial relatively compact for $\varGamma=\langle g\rangle_R$ for some $g\not\in G_{rc}$; and further $\pi\colon\X^{\prime\prime}\rightarrow\Y$ is nontrivial relatively compact for $\varGamma\times G_{rc}$ by Lemma~\ref{lem2.5}. This is a contradiction to the maximality of $G_{rc}$.

Because $R$ is a syndetic ring, $\langle g\rangle_R\cap G_{rc}=\{I\}$ for any $g\not\in G_{rc}$. Otherwise, for some $t\not=0$, $tg\in G_{rc}$, so $Rtg\subset G_{rc}$ and then $\pi\colon\X^\prime\rightarrow\Y$ is relatively compact for $Rtg=\langle tg\rangle_R$. Since $R$ is syndetic, we can find a compact subset $K$ of $R$ with $K+Rt=R$. Further, by Lemma~\ref{lem2.8}, it follows that $\pi\colon\X^\prime\rightarrow\Y$ is relatively compact for $Rg=\langle g\rangle_R$.
This is a contradiction.

Then we can similarly find a maximal $R$-submodule of $G$, say $G_{{rw}}$, outside $G_{rc}$ such that $\pi\colon\X^\prime\rightarrow\Y$ is \textit{jointly} relatively weak-mixing for $G_{{rw}}$. We may claim that
\begin{gather*}
G=G_{rc}\times G_{{rw}}.
\end{gather*}
Indeed, if not, then there exists some $g\in G\setminus G_{rc}\times G_{{rw}}$; clearly $\left(\langle g\rangle_R\times G_{{rw}}\right)\cap G_{rc}=\{I\}$
because
\begin{gather*}
(tg)g_{{rw}}=g_{rc}\in G_{rc},\ t\not=0, g_{{rw}}\in G_{{rw}}\ \Rightarrow\ tg=g_{rc}g_{{rw}}^{-1}\in G_{rc}\times G_{{rw}};
\end{gather*}
so the extension $\pi\colon\X^\prime\rightarrow\Y$ is jointly relatively weak-mixing for $\langle g\rangle_R\times G_{{rw}}$; this contradicts the maximality of $G_{{rw}}$.

Then by $G_{{rw}}\cap G_{rc}=\{I\}$, it follows that $\pi\colon\X^\prime\rightarrow\Y$ is \textit{totally} relatively weak-mixing for $G_{{rw}}$.
Thus, $\pi\colon\X^\prime\rightarrow\Y$ is primitive in the sense of Def.~\ref{def4.1}$(b)$ or Def.~\ref{def4.1}$(c)$.
This proves Theorem~\ref{thm4.2}.
\end{proof}

This theorem will provide us with the main ingredient in the proof of our Structure Theorem~I (Theorem~0.2) as in the classical case of Hillel Furstenberg.

\subsection{Structure theorems}\label{sec4.2}
To formulate the structure theorems, we need the notion of limit of factors. If $\Z=(Z,\mathscr{Z},\mu,G)$ is a measure-preserving system and $\beta_\theta\colon\Z\rightarrow\Z_\theta,\theta\in\Theta$, is a system of factors of $\Z$. We shall say that $\Z$ is an \textit{inverse limit} of the factors $\{\Z_\theta;\theta\in\Theta\}$ if $\mathscr{Z}$ is generated by the family of $\sigma$-subalgebras $\beta_\theta^{-1}[\mathscr{Z}_\theta]$ ($\mu$-mod $0$), i.e., $\mathscr{Z}=\sigma\big{(}\bigcup_{\theta\in\Theta}\beta_\theta^{-1}[\mathscr{Z}]\big{)}$ ($\mu$-mod $0$); written as $\Z=\underleftarrow{\lim}_{\theta\in\Theta}\Z_\theta$.

In addition, any trivial extension is itself totally relatively weak-mixing for $G$. Whence we only need to study nontrivial extensions.

We now reformulate Theorem~0.2 in a slightly general version as follows:

\begin{Thm}[Structure Theorem~I]\label{thm4.3}
Let $G$ be an \textit{lcsc} Noetherian $R$-module over a syndetic ring $(R,+,\cdot)$. Then, for any nontrivial standard Borel extension $\pi\colon\X\rightarrow\Y$, there exists an ordinal $\eta$ and a system of factors $\{\pi_\xi\colon\X\rightarrow\X_\xi\}_{\xi\le\eta}$ with $\pi_\xi=\textit{Id}_X$ and $\{x\}\in\mathscr{X}_\xi\;\forall x\in X$ for $0<\xi<\eta$ such that
\begin{enumerate}
\item[$(a)$] $\X_\eta=\X$ ($\mu$-$\mathrm{mod}$ $0$) and $\X_0=\Y, \pi_0=\pi$.
\item[$(b)$] For each pair of ordinals $\theta,\xi$ with $0\le\theta<\xi\le\eta$, there is a factor $G$-map $\pi_{\xi,\theta}\colon\X_\xi\rightarrow\X_\theta$ such that $\pi_\theta=\pi_{\xi,\theta}\circ\pi_\xi$.
\item[$(c)$] For each ordinal $\xi$ with $0\le\xi<\eta$, $\pi_{\xi+1,\xi}\colon\X_{\xi+1}\rightarrow\X_\xi$ is a nontrivial primitive extension.
\item[$(d)$] If $\xi$ is a limit ordinal with $0<\xi\le\eta$, then $\X_\xi=\underleftarrow{\lim}_{\theta<\xi}\X_\xi$.
\end{enumerate}
Moreover, in the factors chain
\begin{gather*}
\X\rightarrow\dotsm\rightarrow\X_{\xi+1}\xrightarrow[]{\pi_{\xi+1,\xi}}\X_\xi\rightarrow\dotsm\rightarrow\X_1\xrightarrow[]{\pi_{1,0}}\X_0
\end{gather*}
every intermediate link $\pi_{\xi+1,\xi}\colon\X_{\xi+1}\rightarrow\X_\xi$ is such that $\pi_{\xi+1,\xi}=\textit{Id}_X$ for $0<\xi<\eta$.
\end{Thm}

\begin{proof}
If $\pi\colon\X\rightarrow\Y$ itself is primitive, then let $\eta=1, \X_1=\X, \pi_1=\textit{Id}_X$, and $\pi_{1,0}=\pi$ and hence Theorem~\ref{thm4.3} holds. Next assume that $\pi\colon\X\rightarrow\Y$ is not primitive and we will denote by $\Theta=\{1,2,\dotsc,\omega,\omega+1,\dotsc\}$ the set of all ordinals bigger than $0$.

By contrary, suppose that there exists no an ordinal $\eta$ satisfying the requirements of the structure theorem. Then by transfinite induction and letting $\X_0=\Y$, we can construct a system of factors
$$
\pi_0\colon \X\rightarrow\X_0,\quad \pi_\theta\colon\X\rightarrow\X_\theta=(X,\mathscr{X}_\theta,\mu,G)\quad \forall \theta\in\Theta,
$$
such that:
\begin{enumerate}
\item[$(b)$] $\pi_\theta\colon\X\xrightarrow[]{\pi_\xi}\X_\xi\xrightarrow[]{\pi_{\xi,\theta}}\X_\theta$ for all $0\le\theta<\xi$;
\item[$(c)$] for each ordinal $\theta\ge0$, the extension $\pi_{\theta+1,\theta}\colon\X_{\theta+1}\rightarrow\X_\theta$ is primitive and nontrivial; and
\item[$(d)$] $\X_\xi=\underleftarrow{\lim}_{\theta<\xi}\X_\theta$ if $\xi$ is a limit ordinal.
\end{enumerate}

Indeed, let $\pi_0=\pi$; and for the ordinal $\theta=1$, by Theorem~\ref{thm4.2}, it follows that there is a factor $\X_1=(X,\mathscr{X}_1,\mu,G)$ of $\X$ such that
$\pi\colon\X\xrightarrow{\pi_1}\X_1\xrightarrow{\pi_{1,0}}\X_0$ and $\X_1\xrightarrow{\pi_{1,0}}\X_0$ is nontrivial primitive.
Now, given any ordinal $\theta\in\Theta, \theta\ge2$, let, for any ordinal $\xi<\theta$, there be factor $\X_\xi$ of $\X$ such that ${\pi_{\xi+1,\xi}}\colon\X_{\xi+1}\rightarrow\X_\xi$ is nontrivial primitive whenever $\xi+1<\theta$ and that for any $\xi<\alpha<\theta$ there holds
$\pi_\xi\colon\X\xrightarrow{\pi_\alpha}\X_\alpha\xrightarrow{\pi_{\alpha,\xi}}\X_\xi$. If $\theta$ is an isolated ordinal, then by Theorem~\ref{thm4.2} we can interpolate a factor, say $\X_{\theta}=(X,\mathscr{X}_{\theta},\mu,G)$, of $\X$ between $\X\xrightarrow{\pi_{\theta-1}}\X_{\theta-1}$ such that
\begin{gather*}
\pi_{\theta-1}\colon\X\xrightarrow[]{\pi_\theta}\X\xrightarrow[]{\pi_{\theta,\theta-1}}\X_{\theta-1}
\end{gather*}
and such that ${\pi_{\theta,\theta-1}}\colon\X_{\theta}\rightarrow\X_{\theta-1}$ is nontrivial primitive. If $\theta$ is a limit ordinal, then we set $\mathscr{X}_{\theta}=\sigma\big{(}\bigcup_{\xi<\theta}\mathscr{X}_{\xi}\big{)}$ which is a $\sigma$-subalgebra of $\mathscr{X}$, where $\mathscr{X}_0=\pi^{-1}[\mathscr{Y}]$, and we now define $\pi_{\theta,0}=\pi, \pi_{\theta,\xi}=\textit{Id}_X$ for $0<\xi<\theta$. This completes the induction-hypothesis.

Finally we let $\widetilde{\mathscr{X}}_\mu$ be the Boolean $\sigma$-algebra of $(X,\mathscr{X},\mu)$; that is, $\widetilde{\mathscr{X}}_\mu$ consists of equivalence classes of sets in $\mathscr{X}$, where $A\sim B$ if $\mu(A\cup B-A\cap B)=0$. By $\big{|}\widetilde{\mathscr{X}}_\mu\big{|}$ we denote the power of the set $\widetilde{\mathscr{X}}_\mu$ and let $\eta$ be the initial ordinal corresponding to the power $2^{\big{|}\widetilde{\mathscr{X}}_\mu\big{|}}$. Now for any ordinal $\theta<\eta$, since the extension
$\pi_{\theta+1,\theta}\colon\X_{\theta+1}\rightarrow\X_\theta$ is nontrivial, hence we can find a point
$$\tilde{x}_\theta\in\widetilde{\pi_{\theta+1}^{-1}[\mathscr{X}_{\theta+1}]}_\mu-\widetilde{\pi_{\theta}^{-1}[\mathscr{X}_{\theta}]}_\mu\subsetneq\widetilde{\mathscr{X}}_\mu.$$
From that $\tilde{x}_\xi\not=\tilde{x}_\theta$ for all $\xi\not=\theta$, it follows that the power $\left|\{\tilde{x}_\theta;\theta<\eta\}\right|=2^{\big{|}\widetilde{\mathscr{X}}_\mu\big{|}}>\big{|}\widetilde{\mathscr{X}}_\mu\big{|}$. However, this yields a contradiction to that $\{\tilde{x}_\theta; \theta<\eta\}\subset\widetilde{\mathscr{X}}_\mu$.

The proof of Theorem~\ref{thm4.3} is thus completed.
\end{proof}

Recall that an ordinal $\theta$ is said to be isolated if it is not a limit ordinal; in other word, we have the ordinal $\theta-1$. The following is a simple observation.

\begin{Lem}\label{lem4.4}
Let $\varGamma$ be a submodule of an \textit{lcsc} $R$-module $G$ and $\X$ a standard Borel $G$-space. If there is an isolated ordinal $\theta$ and a system of $G$-factors $\{\pi_\xi\colon\X\rightarrow\X_\xi\}_{\xi\le\theta}$ such that
\begin{enumerate}
\item[$(1)$] for each ordinal $\xi<\theta$ there is a factor $G$-map $\pi_{\theta,\xi}\colon\X_\theta\rightarrow\X_\xi$ with $\pi_\xi=\pi_{\theta,\xi}\circ\pi_\theta$, and
\item[$(2)$] $\pi_{\theta,\theta-1}\X_\theta\rightarrow\X_{\theta-1}$ is relatively $\mathbf{C}_4$-compact for $\varGamma$;
\end{enumerate}
then $\pi_{\xi+1,\xi}\X_{\xi+1}\rightarrow\X_\xi$ is relatively $\mathbf{C}_4$-compact for $\varGamma$ for each ordinal $\xi<\theta$.
\end{Lem}

\begin{proof}
The statement follows from the following factors series:
\begin{equation*}
\begin{CD}
\X_\theta@>{\pi_{\theta,\theta-1}}>>\X_{\theta-1}@>{\pi_{\theta-1,\xi+1}}>>\X_{\xi+1}@>{\pi_{\xi+1,\xi}}>>\X_\xi.
\end{CD}
\end{equation*}
Given any $\varphi\in\mathfrak{L}^2(X,\mathscr{X}_{\xi+1},\mu)$, by the definition of $\mathbf{C}_4$, it follows that $\varphi$ belongs to $L_{\textit{FKaap}}^2(X,\mathscr{X}_\theta,G:\varGamma)$. This proves the lemma.
\end{proof}

The following is a special case of Theorem~0.2.

\begin{Thm}[Maximal Distal Factor]\label{thm4.5}
Let $G$ be an \textit{lcsc} Noetherian $R$-module of rank $1$ over a syndetic ring $R$ and $\X$ a standard Borel $G$-space. Then there exists an ordinal $\eta$ and a system of factors $\{\pi_\xi\colon\X\rightarrow\X_\xi\}_{\xi\le\eta}$ such that
\begin{enumerate}
\item[$(a)$] $\pi_\eta\colon \X\rightarrow\X_\eta$ is totally relatively weak-mixing for $G$ or $\X_\eta=\X$ ($\mu$-$\mathrm{mod}$ $0$).
\item[$(b)$] $\X_0\approx(X,\mathscr{X}_0,\mu,G)$ where $\mathscr{X}_0=\{\varnothing,X\}$.
\item[$(c)$] For each pair of ordinals $\theta,\xi$ with $0\le\theta<\xi\le\eta$, there is a factor $G$-map $\pi_{\xi,\theta}\colon\X_\xi\rightarrow\X_\theta$ such that $\pi_\theta=\pi_{\xi,\theta}\circ\pi_\xi$.
\item[$(d)$] For each ordinal $\xi<\eta$, $\pi_{\xi+1,\xi}\colon\X_{\xi+1}\rightarrow\X_\xi$ is a nontrivial relatively compact extension for $G$.
\item[$(e)$] If $\xi$ is a limit ordinal $\le\eta$, then $\X_\xi=\underleftarrow{\lim}_{\theta<\xi}\X_\theta$.
\end{enumerate}
\end{Thm}

\begin{proof}
Based on Lemma~\ref{lem4.4} and Theorems~\ref{thm4.3} and \ref{thm2.4}, we can conclude this theorem.
\end{proof}

\begin{cor}\label{cor4.6}
Let $\mathbb{T}^2=\mathbb{R}^2/\mathbb{Z}^2$ be the $2$-dimensional torus and let $\mu$ be the standard Haar measure on $\mathbb{T}^2$. Assume $\mathbb{Z}$ acts $\mu$-preserving on $\mathbb{T}^2$ induced by $T\colon(x,y)\mapsto(x,y+x)$ and set $\X=(\mathbb{T}^2,\mathscr{B}_{\mathbb{T}^2},\mu,\mathbb{Z})$. Let
\begin{gather*}
\X\rightarrow\dotsm\rightarrow\X_{\theta+1}\xrightarrow[]{\pi_{\theta+1,\theta}}\X_\theta\rightarrow\dotsm\rightarrow\X_1\xrightarrow[]{\pi_{1,0}}\X_0
\end{gather*}
be the factors chain of $\X$ by Theorem~\ref{thm4.5}. Then there exists at least one intermediate link $\pi_{\theta+1,\theta}\colon\X_{\theta+1}\rightarrow\X_\theta$ which is relatively compact for $\mathbb{Z}$ but there is some $\varphi\in\mathfrak{L}^2(\mathbb{T}^2,\mathscr{X}_{\theta+1},\mu)$ with $\varphi\ge0$ a.e. and $\int_{\mathbb{T}^2}\varphi\,d\mu>0$ such that $\varphi$ is not \textit{a.p. for $\mathbb{Z}$} in the sense of Def.~\ref{def2.2}.
\end{cor}

\begin{proof}
Otherwise, by \cite{Dai-pre}, it follows that $\X$ would have the multiple Khintchine recurrence. However, this is a contradiction to \cite[Theorem~1.3 and Theorem~2.1]{BHK}.
\end{proof}

Similar to the proof of Theorem~\ref{thm4.3}, we can obtain the following another structure theorem.

\begin{Thm}[Structure Theorem~II]\label{thm4.7}
Let $G$ be an \textit{lcsc} amenable group and $\X$ a nontrivial standard Borel $G$-space. Then there exists an ordinal $\eta$ and a system of factors $\{\pi_\xi\colon\X\rightarrow\X_\xi\}_{\xi\le\eta}$ such that
\begin{enumerate}
\item[$(a)$] $\pi_\eta\colon \X\rightarrow\X_\eta$ is jointly relatively weak-mixing for $G$ or $\X_\eta=\X$ ($\mu$-$\mathrm{mod}$ $0$).
\item[$(b)$] $\X_0\approx(X,\mathscr{X}_0,\mu,G)$ where $\mathscr{X}_0=\{\varnothing,X\}$.
\item[$(c)$] For each pair of ordinals $\theta,\xi$ with $\theta<\xi\le\eta$, there is a factor $G$-map $\pi_{\xi,\theta}\colon\X_\xi\rightarrow\X_\theta$ such that $\pi_\theta=\pi_{\xi,\theta}\circ\pi_\xi$.
\item[$(d)$] For each ordinal $\xi<\eta$, $\pi_{\xi+1,\xi}\colon\X_{\xi+1}\rightarrow\X_\xi$ is a nontrivial relatively compact extension for $G$ itself.
\item[$(e)$] If $\xi$ is a limit ordinal with $\xi\le\eta$, then $\X_\xi=\underleftarrow{\lim}_{\theta<\xi}\X_\theta$.
\end{enumerate}
\end{Thm}

\begin{proof}
The statement follows from Theorem~\ref{thm3.5} by making use of the transfinite induction as in the proof of Theorem~\ref{thm4.3}.
\end{proof}

We note here that if $G$ is a discrete countable abelian group, this is just another Structure Theorem \cite[Theorem~10.15]{Gla} using different approaches.

In addition, it should be noted that the jointly relatively weak-mixing for $G$ is essentially weaker than the totally relatively weak-mixing for $G$ as mentioned before. It turns out that the totally relatively weak-mixing, not the jointly relatively weak-mixing, is very important for proving the multiple recurrence theorem in our subsequent applications~\cite{Dai-pre, Dai3}.

\section*{\textbf{Acknowledgments}}%
This work was partly supported by National Natural Science Foundation of China grant $\#$11271183 and PAPD of Jiangsu Higher Education Institutions.



\end{document}